\documentclass[11pt]{article}
\usepackage{}

\usepackage[total={6in, 8.5in}]{geometry}
\usepackage{amsfonts}
\usepackage{amsthm}
\usepackage{amssymb}
\usepackage{amsmath}
\usepackage{caption}
\usepackage{enumerate}
\usepackage[
pdfauthor={},
pdftitle={The Core Conjecture of Hilton and Zhao I: Pseudo-multifan and Lollipop},
pdfstartview=XYZ,
bookmarks=true,
colorlinks=true,
linkcolor=blue,
urlcolor=blue,
citecolor=blue,
bookmarks=false,
linktocpage=true,
hyperindex=true
]{hyperref}

\makeatletter
\newcommand{\setword}[2]{%
	\phantomsection
	#1\def\@currentlabel{\unexpanded{#1}}\label{#2}%
}
\makeatother

\makeatletter
\renewcommand*\env@matrix[1][*\c@MaxMatrixCols c]{%
	\hskip -\arraycolsep
	\let\@ifnextchar\new@ifnextchar
	\array{#1}}
\makeatother

\usepackage{array}
\usepackage{graphicx}
\usepackage[natural]{xcolor}

\usepackage{pgf,tikz,pgfplots}

\usepackage{mathrsfs}

\usetikzlibrary{arrows}

\usepackage{xcolor}
\long\def\ignore#1{}

\let\oldi\ignore

 \oldi{
\newtheorem{THM}{\textbf{Theorem}}[section]
\newtheorem{THMs}{\textbf{Theorem}}[section]
\newtheorem{DEF}[THM]{\textbf{Definition}}[section]
\newtheorem{LEM}[THM]{\textbf{Lemma}}
\newtheorem{CON}[THM]{\textbf{Conjecture}}
\newtheorem{PROP}[THM]{\textbf{Proposition}}
\newtheorem{COR}[THM]{\textbf{Corollary}}
\newtheorem{CORs}{\textbf{Corollary}}[section]
\newtheorem{PRO}[THM]{\textbf{Problem}}

\newcommand{\pf}{\textbf{Proof}.\quad}

\newtheorem{FAC}{\textbf{Fact}}
\newtheorem{REM}{\textbf{Remark}}
\newtheorem{OPR}{\textbf{Operation}}
\newtheorem{CLA}{\textbf{Claim}}[section]
\setcounter{CLA}{1}
 }

\newtheorem{THM}{Theorem}[section]

\newtheorem{DEF}[THM]{Definition}
\newtheorem{LEM}[THM]{Lemma}
\newtheorem{CON}[THM]{Conjecture}

\newtheorem{COR}[THM]{Corollary}

\newtheorem{CLA}{Claim}[section]
\newcommand{\pf}{\textbf{Proof}.\quad}

\newtheorem*{THM2}{\textbf{Theorem 2.5}}
\newtheorem*{THM3}{\textbf{Theorem 2.6}}
\newtheorem*{THM4}{\textbf{Theorem 2.8}}
\newtheorem*{Ass}{\textbf{Assumption}}

\usepackage{thmtools}
\usepackage{thm-restate}

\newcommand{\CC}{\mathcal{C}}

\newcommand{\pbar}{\overline{\varphi}}

\begin{document}
\title{Pseudo-multifan and Lollipop}

\author{%
	Yan Cao\thanks{School of Mathematical Sciences, 
		Dalian University of Technology, Dalian,  Liaoning 116024, China.  \texttt{ycao@dlut.edu.cn}.}
	\quad Guantao Chen\thanks{Department of Mathematics and Statistics, 
		Georgia State University, Atlanta, GA 30302, USA.  \texttt{gchen@gsu.edu}.  }\\
	\quad 
	Guangming Jing\thanks{School of Mathematical and Data Sciences, 
		West Virginia University, Morgantown, WV 26506, USA.  \texttt{gujing@mail.wvu.edu}. }
	\quad 
	Songling Shan\thanks{Department of Mathematics and Statistics, 
		Auburn  Univeristy, Auburn, AL 36849, USA. 
		\texttt{szs0398@auburn.edu}.  }
} 

\date{\today}
\maketitle

 \begin{abstract}
 A simple graph  $G$ with maximum degree $\Delta$ is  \emph{overfull} if $|E(G)|>\Delta \lfloor |V(G)|/2\rfloor$. The \emph{core} of $G$, denoted $G_{\Delta}$, is the subgraph of $G$ induced by its vertices of degree $\Delta$.  Clearly, the chromatic index of $G$ equals  $\Delta+1$ if $G$ is overfull. 
Conversely, Hilton and Zhao in 1996 conjectured  that if $G$ is a simple connected graph with $\Delta\ge 3$ and  $\Delta(G_\Delta)\le 2$, then $\chi'(G)=\Delta+1$  implies that $G$ is overfull or $G=P^*$, where $P^*$ is obtained from the Petersen graph by deleting a vertex (Core Conjecture).   The goal of this paper is to  develop 
the concepts of ``pseudo-multifan'' and ``lollipop'' and study their properties in an edge colored graph. These concepts turn out 
to be   powerful tools in edge coloring graphs with a small core degree.


 \end{abstract}


\section{Introduction}

For two integers $p$ and $q$, let $[p,q]=\{i\in \mathbb{Z}: p\le i\le q\}$.
Let $G$ be a simple graph with maximum degree $\Delta$. 
The \emph{core} of  $G$,  denoted $G_\Delta$, 
is the subgraph of $G$ induced by its vertices of degree $\Delta$. 
Let $k\ge 0$ be an integer. 
An  \emph{edge $k$-coloring} of $G$ is a mapping $\varphi$ from $E(G)$ to 
$[1,k]$, called \emph{colors}, such that  no two adjacent edges receive the same color with respect to $\varphi$.  
The  \emph{chromatic index}  $\chi'(G)$ of $G$ is the smallest  $k$ so that $G$ has an edge $k$-coloring.

In 1960's, Gupta~\cite{Gupta-67}  and, independently, Vizing~\cite{Vizing-2-classes}  proved
that for all graphs $G$,  $\Delta \le \chi'(G) \le \Delta+1$. 
This 
leads to a natural classification of simple  graphs. Following Fiorini and Wilson~\cite{fw},  a graph $G$ is of {\it class 1} if $\chi'(G) = \Delta$ and of \emph{class 2} if $\chi'(G) = \Delta+1$.  Holyer~\cite{Holyer} showed that it is NP-complete to determine whether an arbitrary graph is of class 1.  
Nevertheless, if $|E(G)|>\Delta \lfloor |V(G)|/2\rfloor$,  then we have to color $E(G)$ using exactly  $(\Delta+1)$ colors. Such graphs are  \emph{overfull}.
Thus the containment of an overfull subgraph of the same maximum degree is a sufficient condition for a graph to be class 2. The condition is not 
necessary, as  the  Petersen graph is class 2
but contains 
no $3$-overfull subgraph. By Seymour~\cite{seymour79}, it is polynomial-time  to determine whether $G$ contains 
an overfull subgraph of maximum degree $\Delta$. A fundamental question arising  here is that for what graphs 
this sufficient condition of overfull subgraph containment will also be necessary.

Hilton and Zhao~\cite{MR1395947} in 1996
proposed the following Core Conjecture. 
If true, it implies an easy approach to determine the 
chromatic index for connected graphs $G$ with $\Delta(G_\Delta)\le 2$:  just count the number of 
edges in $G$ if $G\ne P^*$, where $P^*$ is obtained from the Petersen graph by deleting one 
vertex.   

\begin{CON}[Core Conjecture]\label{Core Conjecture}
	Let $G$ be a  simple connected graph with maximum degree $\Delta\ge 3$ and $\Delta(G_\Delta)\le 2$. 
	Then $G$ is class 2 implies that $G$ is overfull or $G=P^*$. 
\end{CON}

A connected class 2 graph $G$ with $\Delta(G_\Delta)\leq 2$ is a \emph{Hilton-Zhao graph (HZ-graph)}.  Clearly, $P^*$ is an HZ-graph  with $\chi'(P^*)=4$ and $\Delta(P^*)=3$. Hence the Core Conjecture is equivalent to the claim that every HZ-graph  $G\not=P^*$ with $\Delta(G)\geq 3$ is overfull. Not much progress has been made since the conjecture was proposed  in 1996. A first breakthrough was achieved in 2003, when Cariolaro and Cariolaro \cite{CariolaroC2003} settled the
base case $\Delta=3$. They proved that $P^*$ is the only HZ-graph with maximum degree $\Delta=3$, an alternative proof was given later by Kr\'al', Sereny, and Stiebitz (see \cite[pp. 67--63]{StiebSTF-Book}). The next case,  $\Delta=4$, was recently solved by Cranston and Rabern \cite{CranstonR2018hilton}, they proved that the only HZ-graph with maximum degree $\Delta=4$ is obtained from the graph $K_5$ with one edge removed. 
The conjecture is wide open for $\Delta \ge 5$.  Our main goal in this paper 
is to develop two new concepts, namely ``pseudo-multifan'' and ``lollipop'' that generalize  previously known 
adjacency lemmas associated with multifans and Kierstead paths. These developments were used
to prove the Core Conjecture~\cite{2108.04399}. Furthermore, we have applied 
these ideas in proving the overfullness of graphs in~\cite{CCS22} and~\cite{2208.04179}, making progress towards the Overfull Conjecture by  Chetwynd and Hilton from 1986~\cite{MR848854}. We believe that 
these concepts and related results will be useful tools in the area of edge colorings.

The remainder of the paper is organized as follows. In next section, we give the classical edge coloring concept of a multifan, and 
then  define  ``pseudo-multifans'' and ``lollipops.'' The main results are  listed as Theorem~\ref{pseudo-fan-ele}, Theorem~\ref{Lem:2-non-adj1*} and Theorem~\ref{Lem:2-non-adj2}.
In Section 3, we provide certain preliminaries and notation. In Section 4, we prove Theorem~\ref{pseudo-fan-ele}.  Theorem~\ref{Lem:2-non-adj1*} and Theorem~\ref{Lem:2-non-adj2}
will be proved in the last section.

\section{Multifan, pseudo-multifan, and lollipop}

We start with some  definitions. 
Let $G$ be a graph, $v\in V(G)$, and $i\ge 0$ be an integer.  
An \emph{$i$-vertex}  is a vertex of degree $i$
in $G$, and an $i$-vertex from the neighborhood of   $v$ is called an \emph{$i$-neighbor}  of $v$. 
 Define 
$$ V_i=\{w\in V(G)\,:\, d_G(w)=i\} , \quad \quad N_{i}(v)=N_G(v)\cap V_i, \quad \,\mbox{and} \quad N_i[v]=N_i(v)\cup \{v\}. $$
The symbol $\Delta$  is reserved for $\Delta(G)$, the maximum degree of $G$
throughout  this paper.

Let   $e\in E(G)$ and $\varphi\in \CC^k(G-e)$ for some $k\ge 0$, where $\CC^k(G)$ denotes the set of all edge $k$-colorings of $G$.  
The set of colors \emph{present} at $v$ is 
$\varphi(v)=\{\varphi(f)\,:\, \text{$f$ is incident to $v$}\}$, and the set of colors \emph{missing} at $v$ is $\pbar(v)=[1,k]\setminus\varphi(v)$.  If $\pbar(v)=\{\alpha\}$ is a singleton for some $\alpha\in [1,k]$, we  also write $\pbar(v)=\alpha$. 
For a vertex set $X\subseteq V(G)$,  define 
$
\pbar(X)=\bigcup _{x\in X} \pbar(x).
$
The set $X$ is  \emph{$\varphi$-elementary} if $\pbar(x)\cap \pbar(y)=\emptyset$
for any distinct  $x,y\in X$.    

Let $\alpha,\beta \in [1,k]$. Each component of $G-e$
induced on edges colored by $\alpha$ or $\beta$ is either a 
path or an even cycle,  which is called an \emph{$(\alpha,\beta)$-chain} of $G-e$
with respect to $\varphi$. 
Interchanging  $\alpha$ and $\beta$
on an $(\alpha,\beta)$-chain $C$ of $G$ gives a new edge $k$-coloring, which is denoted by 
$\varphi/C$. 
This operation  is called a \emph{Kempe change}. 

For $x,y\in V(G)$, if $x$ and $y$
are contained in the same  $(\alpha,\beta)$-chain with respect to $\varphi$, we say $x$ 
and $y$ are \emph{$(\alpha,\beta)$-linked}.
Otherwise, they are \emph{$(\alpha,\beta)$-unlinked}. 
If an $(\alpha,\beta)$-chain  $P$ is a path with one endvertex as $x$, we also denote it by $P_x(\alpha,\beta,\varphi)$, and we just write $P_x(\alpha,\beta)$ if $\varphi$ is understood.   For a vertex $u$ and an edge $uv$ contained in $P_x(\alpha,\beta,\varphi)$, 
we write 
{$\mathit {u\in P_x(\alpha,\beta, \varphi)}$} and  {$\mathit {uv\in P_x(\alpha,\beta, \varphi)}$}.  
If $u,v\in P_x(\alpha,\beta,\varphi)$ such that $u$ lies between $x$ and $v$ on $P$, 
then we say that $P_x(\alpha,\beta,\varphi)$ \emph{meets $u$ before $v$}.

Let   
$T$ be  an alternating sequence of vertices  and edges of  $G$. We denote by \emph{$V(T)$}  
the set of vertices  contained in  $T$, and by  
\emph{$E(T)$}  the set of edges contained in $T$. We simply write $\pbar(T)$ for $\pbar(V(T))$. 
If $V(T)$ is $\varphi$-elementary and $\pbar(T) \ne \emptyset$,
then for a color  $\tau\in \pbar(T)$,  we denote by  $\mathit{\pbar^{-1}_T(\tau)}$ the  unique vertex  in $V(T)$ at which $\tau$ 
is missed.  A coloring $\varphi'\in \CC^k(G-e)$
is  \emph{$(T,\varphi)$-stable} if for every  $x\in V(T)$ and every $f\in E(T)$, it holds that $\pbar'(x)=\pbar(x)$ and  $\varphi'(f)=\varphi(f)$.  Clearly, $\varphi$ is 
$(T,\varphi)$-stable.  

Let   $rs_1\in E(G)$ and $\varphi\in \CC^k(G-rs_1)$ for some $k\ge 0$. We now are ready to give the definitions of multifans and pseudo-multifans.

\begin{DEF}[Multifan]
A multifan centered at $r$ with respect to $rs_1$ and $\varphi$
is a sequence $F_\varphi(r,s_1:s_p):=(r, rs_1, s_1, rs_2, s_2, \ldots, rs_p, s_p)$ with $p\geq 1$ consisting of  distinct vertices and edges  such that  for every edge $rs_i$ with $i\in [2,p]$,  there is a vertex $s_j$ with $j\in [1,i-1]$ satisfying 
	$\varphi(rs_i)\in \pbar(s_j)$. 
\end{DEF}

A multifan $F_\varphi(r,s_1:s_p)$ is  \emph{maximum} at $r$ if $|V(F)|$ is maximum among all multifans at $r$. 

\begin{DEF}[Pseudo-multifan]
	A  pseudo-multifan  with respect to $rs_1$ and $\varphi$ is a sequence $S:=S_\varphi(r,s_1:s_t:s_p):=(r, rs_1, s_1, rs_2, s_2, \ldots,rs_t, s_t, rs_{t+1},  s_{t+1}, \ldots, s_{p-1},  rs_p, s_p)$ 
	with $t,p \ge 1$  consisting of distinct  vertices  and edges  satisfying the following conditions:
	\begin{enumerate}[(P1)]
		\item the subsequence $F:=(r, rs_1, s_1, rs_2, s_2, \ldots,rs_t, s_t)$ is a maximum multifan at $r$.
		\item $V(S)$ is $\varphi'$-elementary  for every $(F,\varphi)$-stable $\varphi'\in \CC^k(G-rs_1)$.
	\end{enumerate}
\end{DEF}

	Let $F_\varphi(r,s_1:s_p)$  be a multifan.  We call $s_{\ell_1},s_{\ell_2}, \ldots, s_{\ell_h}$, a subsequence of $s_2, \ldots, s_p$, an  \emph{$\alpha$-inducing sequence} for some $\alpha\in[1,k]$ with respect to $\varphi$ and $F$ if 
	$
	\varphi(rs_{\ell_1})= \alpha\in \pbar(s_1)$ and  $\varphi(rs_{\ell_i})\in \pbar(s_{\ell_{i-1}})$ for each  $i\in [2,h].
	$ (By this definition, $(r, rs_1, s_1, rs_{\ell_1}, s_{\ell_1}, \ldots, rs_{\ell_h}, s_{\ell_h})$ is also a multifan with respect to $rs_1$ and $\varphi$.)
	A  color in $\pbar(s_{\ell_i})$ for any $i\in[1,h]$ is an \emph{$\alpha$-inducing color} and is \emph{induced by} $\alpha$.   For $\alpha_i\in \pbar(s_{\ell_i})$
	and $\alpha_j\in \pbar(s_{\ell_j})$ with $i<j$ and $i,j\in [1,h]$, we write {$\mathit \alpha_i \prec \alpha_j$}.  For convenience, $\alpha$ itself is also an $\alpha$-inducing color and is induced by $\alpha$, and $\alpha\prec  \beta$
	for any $\beta \in \pbar(s_{\ell_i})$ and any $i\in [1,h]$. An $\alpha$-inducing color $\beta$ is called a \emph{last $\alpha$-inducing color} if  there does not exist any $\alpha$-inducing color $\delta$ such that $\beta \prec \delta$.
	
	 An edge
	$e\in E(G)$ is a \emph{critical edge}  if $\chi'(G-e)<\chi'(G)$, and  
	 $G$ is  {\it edge $\Delta$-critical} or simply  \emph{$\Delta$-critical}  if $G$ is connected,  $\chi'(G)=\Delta+1$,  and every edge of $G$ is critical. 
  The following result by Hilton and Zhao in~\cite{MR1172373} indicates that $V(G)=V_\Delta\cup V_{\Delta-1}$ for any HZ-graph $G$ with $\Delta\ge 3$.  
	
	\begin{LEM}\label{biregular}
		If $G$ is  an HZ-graph with maximum degree $\Delta$, then  the following statements hold.
		\begin{enumerate}[(a)]
			\item $G$ is $\Delta$-critical and $G_\Delta$ is 2-regular. 
			\item $\delta(G)=\Delta-1$, or $\Delta=2$ and $G$ is an odd cycle. 
			\item Every vertex of $G$ has at least two neighbors in $G_\Delta$. 
		\end{enumerate}
	\end{LEM}

By Lemma~\ref{biregular} (a), every edge of an HZ graph is critical. 
For an HZ-graph $G$ with maximum degree $\Delta\ge 3$, we let $rs_1\in E(G)$ with $r\in V_\Delta$ and $s_1\in N_{\Delta-1}(r):=\{s_1,s_2,\dots, s_{\Delta-2}\}$, and $\varphi\in \CC^\Delta(G-rs_1)$. Then we call $(G,rs_1,\varphi)$ a \emph{coloring-triple}. 
As $\Delta$-degree vertices in a multifan do not miss any color, 
for  multifans in  HZ-graphs, we add a further requirement in its definition as follows 
and we use this new definition  in HZ-graphs in the remainder.
\begin{Ass}
For multifans in  HZ-graphs,	all of its vertices except the center  have degree $\Delta-1$.   
\end{Ass}

Let $(G,rs_1,\varphi)$ be a coloring-triple and $F:=F_\varphi(r,s_1:s_p)$ be a multifan. By its definition,  $|\pbar(s_1)|=2$, $|\pbar(s_i)|=1$ for each $i\in [2,p]$, and  so every color in $\pbar(F)\setminus \pbar(r)$ is induced by one of the two colors in $\pbar(s_1)$. 
We call $F$ a \emph{typical multifan}, denoted $F_\varphi(r, s_1:s_\alpha:s_\beta):=(r, rs_1, s_1, rs_2, s_2, \ldots,rs_\alpha, s_\alpha, rs_{\alpha+1}, s_{\alpha+1}, \ldots, rs_\beta, s_\beta)$ where $\beta:=p$, 
\begin{itemize}
	\item  if  $\pbar(r)=1$ (recall we denote $\pbar(v)$ by a number if $|\pbar(v)|=1$) and $\pbar(s_1)=\{2,\Delta\}$;  and 
	\item  if $|V(F)|\ge 3$, then $\varphi(rs_{\alpha+1})=\Delta$ and $\pbar(s_{\alpha+1})=\alpha+2$ (if $\beta>\alpha$), and 
	for each $i\in [2,\beta]$ with $i\ne \alpha+1$, $\varphi(rs_i)=i$  and $\pbar(s_i)=i+1$.	
\end{itemize}
It is clear that $ s_2, \ldots, s_\alpha$ is the longest 
$2$-inducing sequence  and $s_{\alpha+1}, \ldots, s_\beta$ (if $\beta>\alpha$) is the longest 
$\Delta$-inducing sequence of $F_\varphi(r, s_1:s_\alpha:s_\beta)$. 
By relabelling   vertices and colors if necessary, any multifan in an HZ-graph can be assumed to be a typical multifan, see Figure~\ref{f1} (a) for a depiction.  If $\alpha=\beta$, then we write $F_\varphi(r,s_1:s_\alpha)$ for $F_\varphi(r, s_1:s_\alpha:s_\beta)$, and call it  a {\it typical 2-inducing multifan}.  

If $F=(a_1,\ldots, a_t )$ is a sequence, then for a new 
entry $b$, 
$(F, b)$  denotes the sequence $(a_1,\ldots, a_t, b)$.

\begin{DEF}[Lollipop]
Let $(G,rs_1,\varphi)$ be a coloring-triple. A {lollipop}
centered at $r$  is a sequence 
$L=(F, ru, u,ux, x)$ of distinct vertices and edges such that $F=F_\varphi(r,s_1:s_\alpha:s_\beta)$ is a typical multifan, $u\in N_\Delta(r)$ and $x\in N_{\Delta-1}(u)$ 
with $x\not\in\{s_1,\ldots, s_\beta\}$ (see Figure~\ref{f1} (b) for a depiction).
\end{DEF}

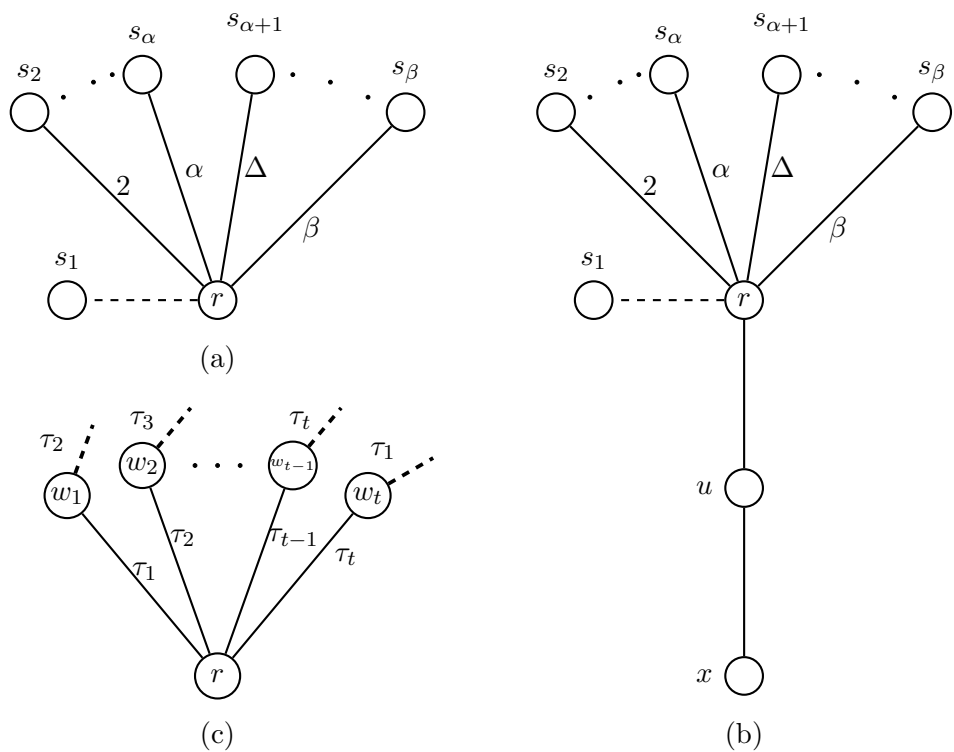
\begin{figure}[!htb]
	\begin{center}
		\begin{tikzpicture}[scale=1]
		
		{\tikzstyle{every node}=[draw ,circle,fill=white, minimum size=0.5cm,
			inner sep=0pt]
			\draw[black,thick](-1,0) node[label={below: }] (r)  {$r$};
			\draw[black,thick](-3,0) node[label={above: $s_1$}] (s1) {};
			\draw[black,thick](-3.5,2.5) node[label={above: $s_2$}] (s2) {};
			\draw[black,thick](-2,3) node[label={above: $s_\alpha$}] (sa) {};
			\draw[black,thick](-0.5,3) node[label={above: $s_{\alpha+1}$}] (sa1) {};
			\draw[black,thick](1.5,2.5) node[label={above: $s_\beta$}] (sb) {};
			
		}
		\path[draw,thick,black, dashed]
		(r) edge node[name=la,above,pos=0.5] {\color{black}} (s1);
		
		\path[draw,thick,black]
		(r) edge node[name=la,above,pos=0.5] {\color{black}$2$} (s2)
		(r) edge node[name=la,above,pos=0.5] {\color{black}\quad$\alpha$} (sa)
		(r) edge node[name=la,above,pos=0.5] {\color{black}\quad\,\,$\Delta$} (sa1)
			(r) edge node[name=la,below,pos=0.5] {\color{black}$\beta$} (sb);

	{\tikzstyle{every node}=[draw ,circle,fill=black, minimum size=0.05cm,
		inner sep=0pt]
		\draw(-3.05,2.7) node (f1)  {};
		\draw(-2.65,2.9) node (f1)  {};
		\draw(-2.4,3) node (f1)  {};
			\draw(0,3) node (f1)  {};
		\draw(0.5,2.9) node (f1)  {};
		\draw(1,2.7) node (f1)  {};
			} 
	\draw(-1,-0.8) node (f1)  {(a)};
	\begin{scope}[shift={(7,0)}]
	{\tikzstyle{every node}=[draw ,circle,fill=white, minimum size=0.5cm,
		inner sep=0pt]
		\draw[black,thick](-1,0) node[label={left: }] (r)  {$r$};
		\draw[black,thick](-3,0) node[label={above: $s_1$}] (s1) {};
		\draw[black,thick](-3.5,2.5) node[label={above: $s_2$}] (s2) {};
		\draw[black,thick](-2,3) node[label={above: $s_\alpha$}] (sa) {};
		\draw[black,thick](-0.5,3) node[label={above: $s_{\alpha+1}$}] (sa1) {};
		\draw[black,thick](1.5,2.5) node[label={above: $s_\beta$}] (sb) {};
			\draw[black,thick](-1,-2.5) node[label={left: $u$}] (u)  {};
				\draw[black,thick](-1,-5) node[label={left: $x$}] (x)  {};
	}
	\path[draw,thick,black, dashed]
	(r) edge node[name=la,above,pos=0.5] {\color{black}} (s1);
	
	\path[draw,thick,black]
	(r) edge node[name=la,above,pos=0.5] {\color{black}$2$} (s2)
	(r) edge node[name=la,above,pos=0.5] {\color{black}\quad$\alpha$} (sa)
	(r) edge node[name=la,above,pos=0.5] {\color{black}\quad\,\,$\Delta$} (sa1)
	(r) edge node[name=la,below,pos=0.5] {\color{black}$\beta$} (sb)
	(r) edge node[name=la,below,pos=0.5] {\color{black}} (u)
	(u) edge node[name=la,below,pos=0.5] {\color{black}} (x);
	
		{\tikzstyle{every node}=[draw ,circle,fill=black, minimum size=0.05cm,
		inner sep=0pt]
		\draw(-3.05,2.7) node (f1)  {};
		\draw(-2.65,2.9) node (f1)  {};
		\draw(-2.4,3) node (f1)  {};
		\draw(0,3) node (f1)  {};
		\draw(0.5,2.9) node (f1)  {};
		\draw(1,2.7) node (f1)  {};
	} 
	
	\draw(-1,-5.8) node (f1)  {(b)};
	
	\end{scope}	
	
	\begin{scope}[shift={(-1,-2)}]
	{\tikzstyle{every node}=[draw ,circle,fill=white, minimum size=0.6cm,
		inner sep=0pt]
		\draw[black,thick] (0, -3) node (r)  {$r$};
		\draw[black,thick] (-2, 0.4-1) node (sa)  {$w_1$};
		\draw [black,thick](-1, 0.8-1) node (sa2)  {$w_2$};
		\draw [black,thick](1, 0.8-1) node (sb)  {\tiny$w_{t-1}$};
		\draw [black,thick](2, 0.4-1) node (sb2)  {$w_t$};
	}
	\path[draw,thick,black]
	(r) edge node[name=la,pos=0.6] {\color{black}\quad$\tau_1$} (sa)
	(r) edge node[name=la,pos=0.7] {\color{black}\quad$\tau_2$} (sa2)
	(r) edge node[name=la,pos=0.7] {\color{black}\quad\,\,\,\,\,$\tau_{t-1}$} (sb)
	(r) edge node[name=la,pos=0.7] {\color{black}\qquad$\tau_t$} (sb2);

	\draw[dashed, black, line width=0.5mm] (sa)--++(70:1cm); 
	\draw[dashed, black, line width=0.5mm] (sa2)--++(50:1cm); 
	\draw[dashed, black, line width=0.5mm] (sb)--++(50:1cm); 
	\draw[dashed, black, line width=0.5mm] (sb2)--++(30:1cm);

	\draw[black] (-2.2, 1.1-1) node {$\tau_2$};  
	\draw[black] (-1.0, 1.4-1) node {$\tau_3$};  
	\draw[black] (1.1, 1.4-1) node {$\tau_{t}$}; 
	\draw[black] (2.2, 1.0-1) node {$\tau_{1}$}; 
	
	{\tikzstyle{every node}=[draw ,circle,fill=black, minimum size=0.05cm,
		inner sep=0pt]
		
		\draw(-0.3,0.8-1) node (f1)  {};
		\draw(0,0.8-1) node (f1)  {};
		\draw(0.3,0.8-1) node (f1)  {};

	} 
	
	\draw(0,-3.8) node (f1)  {(c)};
	\end{scope}
		\end{tikzpicture}
	\end{center}
	\caption{(a) A typical multifan $F_\varphi(r, s_1:s_\alpha:s_\beta)$, where $\pbar(r)=1$ and $\pbar(s_1)=\{2,\Delta\}$;   (b) A lollipop centered at $r$, where $x$ can be the same as some $s_\ell$ for $\ell\in [\beta+1, \Delta-2]$; (c)  A rotation centered at $r$, where a dashed line at a vertex indicates a color missing at the vertex.}
		\label{f1}
		\vspace{-0.2cm}
\end{figure}

Let $(G,rs_1,\varphi)$ be a coloring-triple.
A sequence of distinct vertices 
$w_1,  \ldots, w_t \in N_{\Delta-1}(r)$  form a \emph{rotation}  if 
 $\{w_1,\ldots, w_t\}$ is $\varphi$-elementary, and 
 for each $\ell$ with $\ell\in [1,t]$, it holds that  $\varphi(rw_\ell)=\pbar(w_{\ell-1})$, where $w_0:=w_t$. 
An example of a rotation is given in Figure~\ref{f1} (c).

For $u,v\in V(G)$, we write $u\sim v$ if $u$ and $v$ are adjacent in $G$, and write $u\not\sim v$ otherwise. The main results of this paper are the following.

\begin{THM}\label{pseudo-fan-ele}
	Let  $(G,rs_1,\varphi)$ be a coloring-triple, $S:=S_\varphi(r, s_1: s_t: s_{\Delta-2})$  be a pseudo-multifan 
	 with $F:=F_\varphi(r,s_1:s_t)$ being  the maximum multifan contained in it. Let	$j\in [t+1,\Delta-2]$ and $\delta \in \pbar(s_j)$.   Then 
	\begin{enumerate}[(a)]
		\item $\{s_{t+1}, \ldots, s_{\Delta-2}\}$ can be partitioned 
		into rotations with respect to $\varphi$. \label{pseudo-a}

		\item $s_j$ and $r$ are $(1,\delta)$-linked with respect to  $\varphi$ . \label{pseudo-a1}
		\item For every color $\gamma\in \pbar(F)$ with $\gamma \ne 1$,  it holds that $r\in P_{y}(\gamma,\delta)=P_{s_j}(\gamma,\delta)$, where $y=\mathit{\pbar_{F}^{-1}(\gamma)}$.  
		Furthermore, for $z\in N_G(r)$ such that $\varphi(rz)=\gamma$,  
		$P_{y}(\gamma,\delta)$	meets $z$ before  $r$.  \label{pseudo-b}
		\item For every $\delta^*\in \pbar(S)\setminus \pbar(F)$ with $\delta^*\ne \delta$, it holds that $P_{y}(\delta,\delta^*)=P_{s_{j}}(\delta,\delta^*)$,  where $y=\mathit{\pbar_{S}^{-1}(\delta^*)}$. Furthermore, either $r\in P_{s_j}(\delta,\delta^*)$ or $P_r(\delta, \delta^*)$ is an even cycle. 
		\label{pseudo-c}
		\end{enumerate}
\end{THM}

\begin{THM}\label{Lem:2-non-adj1*}
	Let  $(G,rs_1,\varphi)$ be a coloring-triple, 
	$F:=F_\varphi(r,s_1:s_\alpha:s_\beta)$ be a typical multifan, and  $L:=(F,ru,u,ux,x)$ be a lollipop centered at $r$.  If $\varphi(ru)=\alpha+1$, $\pbar(x)=\alpha+1$, 
	and $\varphi(ux)=\Delta$, 
	then  the following two statements hold.
	\begin{enumerate}[(1)]
		\item If $u\sim s_1$, then $\varphi(us_1)$ is a $\Delta$-inducing color of $F$. 
		\item If $u\sim s_\alpha$, then $\varphi(us_\alpha)$ is a $\Delta$-inducing color of $F$.
	\end{enumerate}
\end{THM}

Since in a typical 2-inducing multifan, $\Delta\in \pbar(s_1)$ is the only $\Delta$-inducing color,  we 
have the following consequence of Theorem~\ref{Lem:2-non-adj1*}. 
\begin{COR}\label{Lem:2-non-adj1}
	Let  $(G,rs_1,\varphi)$ be a coloring-triple,
	$F:=F_\varphi(r,s_1:s_\alpha)$ be a typical 2-inducing  multifan, and  $L:=(F,ru,u,ux,x)$ be a lollipop centered at $r$.  If $\varphi(ru)=\alpha+1$, $\pbar(x)=\alpha+1$, 
	and $\varphi(ux)=\Delta$, 
	then  $u\not\sim s_1$ and $u\not\sim s_\alpha$. 
\end{COR}

\begin{THM}\label{Lem:2-non-adj2}
	Let  $(G,rs_1,\varphi)$ be a coloring-triple, 
	$F:=F_\varphi(r,s_1:s_\alpha)$ be a typical 2-inducing  multifan, and  $L:=(F,ru,u,ux,x)$ be a lollipop centered at $r$.  If $\varphi(ru)=\alpha+1$, $\pbar(x)=\alpha+1$, 
	and $\varphi(ux)=\mu\in \pbar(F)$ is a 2-inducing color of $F$, 
	then  $u\not\sim s_{\mu-1}$ and $u\not\sim s_\mu$.  
\end{THM}

\section{Preliminaries}

In this section, we list some known results on multifans and 
introduce further notation. 

\begin{LEM}[{\cite[Theorem~2.1]{StiebSTF-Book}}]
	\label{thm:vizing-fan1}
	Let $G$ be a class 2 graph and $F_\varphi(r,s_1:s_p)$  be a multifan with respect to   $rs_1$ and  $\varphi\in \CC^\Delta(G-rs_1)$. Then  the following statements  hold. 
	\begin{enumerate}[(a)]
		\item $V(F)$ is $\varphi$-elementary. \label{thm:vizing-fan1a}
		\item For any $\alpha\in \pbar(r)$ and any  $\beta\in \pbar(s_i)$ with $i\in [1,p]$,  $r$ 
		and $s_i$ are $(\alpha,\beta)$-linked with respect to $\varphi$. \label{thm:vizing-fan1b}
	\end{enumerate}
\end{LEM}

By Lemma~\ref{thm:vizing-fan1} (a) and the definition of a multifan $F_\varphi(r,s_1:s_p)$, each color in $\pbar(F)\setminus \pbar(r)$ is induced by a unique color in $\pbar(s_1)$. Also if $\alpha_1$ and $\alpha_2$ are two distinct colors in $\pbar(s_1)$, then an $\alpha_1$-inducing sequence is disjoint from  an $\alpha_2$-inducing sequence. 
As a consequence of Lemma~\ref{thm:vizing-fan1} (a), we have the following properties for a multifan. 
\begin{LEM}
	\label{thm:vizing-fan2}
	Let $G$ be a class 2 graph and $F_\varphi(r,s_1:s_p)$  be a multifan with respect to  $rs_1$ and  $\varphi\in \CC^\Delta(G-rs_1)$. For any two colors $\delta, \lambda$ with $\delta\in \pbar(s_i)$ and $\lambda\in \pbar(s_j)$ for some distinct $i,j\in [1,p]$, the following statements  hold.
	\begin{enumerate}[(a)]
		\item If $\delta$ and $\lambda$ are induced by different colors from $\pbar(s_1)$, then $s_i$ and $s_j$ are $(\delta, \lambda)$-linked with respect to $\varphi$. 
		\label{thm:vizing-fan2-a}
		\item If $\delta$ and $\lambda$ are induced by the same color from $\pbar(s_1)$ such that $\delta\prec\lambda$ and $s_i$ and $s_j$ are $(\delta, \lambda)$-unlinked with respect to $\varphi$, 
		then $r\in P_{s_j}(\delta,\lambda, \varphi)$.  	\label{thm:vizing-fan2-b}
	\end{enumerate}
	
\end{LEM}
\begin{proof}
	
	For (a),  suppose otherwise that $s_i$ and $s_j$ are $(\delta,\lambda)$-unlinked with respect to $\varphi$. Assume that $\delta$ and $\lambda$ are induced by $\alpha$ and $\beta$ respectively where $\alpha,\beta \in \pbar(s_1)$ are  distinct.  Let $ s_{i_1}, s_{i_2},\ldots,s_{i_k}=s_i$ be an $\alpha$-inducing sequence containing $s_i$, and $ s_{j_1},s_{j_2},\ldots,s_{j_\ell}=s_j$ be a $\beta$-inducing sequence containing $s_j$. Since $V(F)$ is $\varphi$-elementary, $s_i$ is the only vertex in $F$ 
	that misses $\delta$. Therefore, the other end  
	of $P_{s_j}(\delta,\lambda,\varphi)$ is outside of $V(F)$.  Let $\varphi'=\varphi/ P_{s_j}(\delta,\lambda,\varphi)$. 
	Then $F^*=(r, rs_1, s_1, rs_{i_1},s_{i_1},\ldots,s_{i_k},rs_{j_1},s_{j_1},\ldots,s_{j_\ell})$
	is a multifan under $\varphi'$. However, $\delta\in \pbar'(s_i)\cap \pbar'(s_j)$, contradicting Theorem~\ref{thm:vizing-fan1} (a).  
	
	For (b),  suppose otherwise that $r\not\in P_{s_j}(\delta,\lambda,\varphi)$. 
	Assume, without loss of generality, that  $i<j$, and  $  s_2, \ldots, s_i, s_{i+1}, \ldots, s_j$ is an $\alpha$-inducing sequence for some $\alpha\in \pbar(s_1)$.
	Since $V(F)$ is $\varphi$-elementary, $s_i$ is the only vertex in $F$ 
	that misses $\delta$. Therefore,  when $s_i$ and $s_j$ are $(\delta,\lambda)$-unlinked with respect to $\varphi$, the other end  
	of $P_{s_j}(\delta,\lambda,\varphi)$ is outside of $V(F)$.   
	Let $\varphi'=\varphi/ P_{s_j}(\delta,\lambda,\varphi)$.  
	Since $r\not\in P_{s_j}(\delta,\lambda,\varphi)$, 
	$\varphi'$ agrees with $\varphi$ on $F$ at every edge and every vertex except $s_j$. 
	Therefore,  the sequence $F_{\varphi'}(r,s_1: s_j)$, 
	obtained from $F_\varphi(r,s_1:s_p)$ by deleting every entry after $s_j$
	is still a multifan. However, $\delta\in \pbar'(s_i)\cap \pbar'(s_j)$, contradicting Theorem~\ref{thm:vizing-fan1} (a).  
\end{proof}

Let $(G,rs_1,\varphi)$ be a coloring-triple and $i,j\in [2,\Delta-2]$.  
The \emph{shift from $s_i$ to $s_j$}  is an operation that,  for each $\ell $ with $ \ell\in[i,j]$,  recolor $rs_\ell$ by  the color in $\pbar(s_\ell)$.  We will apply a shift either on a sequence of vertices from a multifan 
or on a rotation. 

Let $\alpha,\beta,\gamma,\tau\in [1,\Delta] $ and $x,y\in V(G)$. If $P$ is an 
$(\alpha,\beta)$-chain  containing both $x$ and $y$ such that $P$ is a path,  
we denote by $\mathit{P_{[x,y]}(\alpha,\beta, \varphi)}$  the subchain  of $P$ that has endvertices $x$
and $y$. 
Suppose  $|\pbar(x)\cap \{\alpha,\beta\}|=1$. Then an \emph{$(\alpha,\beta)$-swap} 
at $x$ is just the Kempe change on $P_x(\alpha,\beta,\varphi)$. By convention, an	$(\alpha,\alpha)$-swap at $x$ does  nothing at $x$. Suppose $\beta_0\in \pbar(x)$
and $\beta_1,\ldots \beta_t\in \varphi(x)$ for colors $\beta_0, \ldots, \beta_t\in [1,\Delta]$
for some integer $t \ge 1$. Then a 
$$
(\beta_0,\beta_1)-(\beta_1,\beta_2)-\ldots-(\beta_{t-1},\beta_t)-\text{swap}
$$  
at $x$ consists of $t$ Kempe changes: let $\varphi_0=\varphi$,
then $\varphi_i=\varphi_{i-1}/P_x(\beta_{i-1},\beta_i, \varphi_{i-1})$ for each $i\in [1,t]$.  
Suppose the current color of an  edge $uv$ of $G$
is $\alpha$, the notation  $\mathit{uv: \alpha\rightarrow \beta}$  means to recolor  the edge  $uv$ using the color $\beta$. 

We will use a  matrix with two rows to denote a sequence of operations  taken based on $\varphi$.
For example,   the matrix below indicates 
three consecutive operations:  
\[
\begin{bmatrix}
P_{[a, b]}(\alpha, \beta,\varphi)  & s_c:s_{d} & rs\\
\alpha/\beta & \text{shift} & \gamma \rightarrow \tau 
\end{bmatrix}.
\]
\begin{enumerate}[Step 1]
	\item Exchange $\alpha$ and $\beta$ on the $(\alpha,\beta)$-subchain $P_{[a, b]}(\alpha, \beta,\varphi) $.
	\item Based on the coloring obtained from  Step 1, shift from $s_c$ to $s_d$
	for vertices $s_c, \ldots, s_d$. 
	
	\item Based on the coloring obtained from  Step 2,  do  $rs: \gamma \rightarrow \tau $. 
\end{enumerate}

In the reminder,  for simpler description, we may skip the phrase ``with respect to $\varphi$'' in related notation, which then  needs to be understood with respect to the current edge coloring.


%
%

%

\section{Proof of Theorem~\ref{pseudo-fan-ele}}
 \begin{THM2}
 Let  $(G,rs_1,\varphi)$ be a coloring-triple, $S:=S_\varphi(r, s_1: s_t: s_{\Delta-2})$  be a pseudo-multifan 
 with $F:=F_\varphi(r,s_1:s_t)$ being  the maximum multifan contained in it. Let	$j\in [t+1,\Delta-2]$ and $\delta \in \pbar(s_j)$.   Then 
  	\begin{enumerate}[(a)]
 		\item $\{s_{t+1}, \ldots, s_{\Delta-2}\}$ can be partitioned 
 		into rotations with respect to $\varphi$. \label{pseudo-a}

 		\item $s_j$ and $r$ are $(1,\delta)$-linked with respect to  $\varphi$ . \label{pseudo-a1}
 		\item For every color $\gamma\in \pbar(F)$ with $\gamma \ne 1$,  it holds that $r\in P_{y}(\gamma,\delta)=P_{s_j}(\gamma,\delta)$, where $y=\mathit{\pbar_{F}^{-1}(\gamma)}$.  
 		Furthermore, for $z\in N_G(r)$ such that $\varphi(rz)=\gamma$,  
 		$P_{y}(\gamma,\delta)$	meets $z$ before  $r$.  \label{pseudo-b}
 		\item For every $\delta^*\in \pbar(S)\setminus \pbar(F)$ with $\delta^*\ne \delta$, it holds $P_{y}(\delta,\delta^*)=P_{s_{j}}(\delta,\delta^*)$,  where $y=\mathit{\pbar_{S}^{-1}(\delta^*)}$. Furthermore, either $r\in P_{s_j}(\delta,\delta^*)$ or $P_r(\delta, \delta^*)$ is an even cycle. 
 		\label{pseudo-c}
 	\end{enumerate}
 \end{THM2}

\pf By relabeling colors and vertices, we assume $F$ is typical and let   $F=F_\varphi(r,s_1:s_\alpha:s_\beta)$,  where  $\beta=t$.

Let $u,v$ be the two $\Delta$-neighbors of $r$ in $G$. 
Then $\pbar(F)\setminus \{1\}=\{\varphi(rs_i): i\in [2,\beta]\}\cup \{\varphi(ru), \varphi(rv)\}$. Thus  $\{\varphi(rs_i)\,:\, i\in [\beta+1, \Delta-2]\}=[1,\Delta]\setminus \pbar(F)$. 
Also $\cup_{i=\beta+1}^{\Delta-2}\pbar(s_i)=[1,\Delta] \setminus \pbar(F)$  since $V(S)$ is $\varphi$-elementary. Therefore,  
$$\cup_{i=\beta+1}^{\Delta-2}\pbar(s_i)=\{\varphi(rs_i)\,:\, i\in [\beta+1, \Delta-2]\}.$$
Thus, the sequence of missing colors  $\pbar(s_{\beta+1}), \ldots, \pbar(s_{\Delta-2})$
is a  permutation  of the sequence of colors $\varphi(rs_{\beta+1}), \ldots, \varphi(rs_{\Delta-2})$.    
Since every permutation can be partitioned into disjoint cycles, 
 $\{{s_{\beta+1}, \ldots, s_{\Delta-2}}\}$  has a partition into rotations. 
This finishes the proof for \eqref{pseudo-a}. 

By statement \eqref{pseudo-a},  there is a rotation containing $s_j$. Assume, without loss of generality, that this rotation is $s_j, s_{j+1}, \ldots, s_\ell$
in the remainder of this proof. 

For  \eqref{pseudo-a1}, if $s_j$ and $r$
 are $(1,\delta)$-unlinked with respect to $\varphi$, then $P_{s_j}(1,\delta)$ ends at a vertex outside $V(F)$ and does not contain any edge in $F$. Thus $\varphi'=\varphi/P_{s_j}(1,\delta)$ is $(F,\varphi)$-stable. But $V(S)$ is not $\varphi'$-elementary, giving a contradiction to (P2) in the definition of a pseudo-multifan.  

For the first part of statement\eqref{pseudo-b}, since $F$ is typical, suppose to the contrary that
there exists $\gamma\in \pbar(s_{\gamma-1})$, $\gamma\in[2, \beta+1]\cup \{\Delta\}$,  such that $r\in P_{s_{\gamma-1}}(\gamma,\delta)=P_{s_j}(\gamma,\delta)$ does not hold, where $s_{\Delta-1}:=s_1$.  As the proof on the $2$-inducing
sequence and $\Delta$-inducing sequence of $F$ are symmetric up to renaming vertices  and 
colors, we assume that 
$\gamma\in [2,\alpha+1]$. We have the following three cases: $r\notin P_{s_{\gamma-1}}(\gamma,\delta)$ and $r\notin P_{s_j}(\gamma,\delta)$; $r\notin P_{s_{\gamma-1}}(\gamma,\delta)$ and $r\in P_{s_j}(\gamma,\delta)$; and $r\in P_{s_{\gamma-1}}(\gamma,\delta)$ and $r\notin P_{s_j}(\gamma,\delta)$.

Suppose first that $r\notin P_{s_{\gamma-1}}(\gamma,\delta)$ and $r\notin P_{s_j}(\gamma,\delta)$. Let $\varphi'=\varphi/Q$,  where $Q$ is the $(\gamma,\delta)$-chain containing $r$. Note that $\varphi'$ and $\varphi$ agree on every edge incident to $r$ except two edges $rs_{j+1}$ and $rz$,  where recall $z\in N_G(r)$ such that $\varphi(rz)=\gamma$. Thus $z=s_{\gamma}$ if $\gamma \le\alpha$. If $\gamma \le\alpha$,  then $\pbar'(s_{\gamma-1})=\gamma=\varphi'(rs_{j+1})$ and   $\pbar'(s_j)=\delta=\varphi'(rs_{\gamma})$. If $\gamma=\alpha+1$ and $\beta>\alpha$, then  $\varphi'(rs_{\gamma})=\varphi(rs_{\gamma})=\Delta \in \pbar'(s_{1})$. 
 Since $r\notin P_{s_{\gamma-1}}(\gamma,\delta, \varphi)$, $r\notin P_{s_j}(\gamma,\delta,\varphi)$ and $N_{\Delta-1}(r)$ is $\varphi$-elementary, $\pbar'(s_i)=\pbar(s_i)$ for each $i\in [1,\Delta-2]$. 
Thus under the new coloring $\varphi'$, $F^*=(r, rs_1, s_1, \ldots,s_{\gamma-1}, rs_{j+1}, s_{j+1}, \ldots, rs_\ell, s_\ell, rs_j, s_j, rs_{\gamma},  s_{\gamma}, \ldots,s_\beta)$ is a multifan, where we remove repeated elements in $F$ if $\gamma=2$ (see Table~\ref{t1} for an illustration of ``connecting'' vertices to form $F^*$ when $\gamma \le \alpha$). However, $|V(F)|<| V(F^*)|$, 
we obtain a contradiction to the maximality of $F$ as  required  in  (P1) of the definition of a pseudo-multifan.  

\begin{table}[ht]
\begin{center}
	\begin{tabular}{ m{3em}|l|l|l|l|l|l|l|l|l|l|l|l|l} 
	 & $s_1$  &   $\ldots$ & $s_{\gamma-1}$ & $s_{\gamma}$ & $\ldots$ & $s_\alpha$& $\ldots$ & $s_\beta$& \ldots & $s_j$& $s_{j+1}$ & $\ldots$ & $s_\ell$ \\
		\hline
		Missing color of $s_i$ & $2,\Delta$   &  $\ldots$ & $\gamma$ & $\gamma+1$ & $\ldots$  & $\alpha+1$& $\ldots$ & $\beta+1$& $\ldots$ & $\delta$& $\delta_1$ & $\ldots$& $\delta_\ell$ \\
		\hline
		Color of $rs_i$ &     &  $\ldots$ & $\gamma-1$ & $\gamma \color{red}{\rightarrow \delta}$ & $\ldots$ & $\alpha$   & $\ldots$ & $\beta$  &  $\ldots$& $\delta_\ell$& $\delta {\color{red}{\rightarrow \gamma}}$ & $\ldots$ & $\delta_{\ell-1}$ \\
		\hline
	\end{tabular}

	\caption{Statement (c): when $r\notin P_{s_{\gamma-1}}(\gamma,\delta)$ and $r\notin P_{s_j}(\gamma,\delta)$, the coloring changes in the neighborhood of $r$ when $2<\gamma \le \alpha$.}
	\label{t1}
\end{center}
\vspace{-0.5cm}
\end{table}

Suppose then that $r\notin P_{s_{\gamma-1}}(\gamma,\delta)$ and $r\in P_{s_j}(\gamma,\delta)$. Let $\varphi'=\varphi/P_{s_j}(\gamma,\delta)$. Similar to the case above, one can easily check that $F^*=(r, rs_1, s_1, \ldots,s_{\gamma-1}, rs_{j+1}, s_{j+1}, \ldots, rs_\ell, s_\ell, rs_j, s_j)$ is a multifan. 
Since $\pbar'(s_{\gamma-1})=\pbar'(s_j)=\gamma$, we obtain a contradiction to 
Lemma~\ref{thm:vizing-fan1} \eqref{thm:vizing-fan1a} that 
 $V(F^*)$ is $\varphi'$-elementary. 

Suppose lastly that $r\in P_{s_{\gamma-1}}(\gamma,\delta)$ and $r\notin P_{s_j}(\gamma,\delta)$. Then let $\varphi'=\varphi/P_{s_j}(\gamma,\delta)$. Note that $\varphi'$ is $(F,\varphi)$-stable.  Thus by the definition of a pseudo-multifan, $V(S)$ is $\varphi'$-elementary. But $\pbar'(s_{\gamma-1})=\pbar'(s_j)=\gamma$, giving a contradiction. This completes the proof of the first part of statement \eqref{pseudo-b}.

For the second  part of statement\eqref{pseudo-b}, assume to the contrary that $P_{y}(\gamma,\delta)$
meets $r$ before  $z$. Then $P_{y}(\gamma,\delta)$ meets $s_{j+1}$ before $r$ as $\varphi(rs_{j+1})=\delta$. Let $\varphi'$ be obtained from $\varphi$ by shift from $s_{j}$
to $s_{\ell}$. Then  $r\not\in P_{y}(\gamma,\delta, \varphi')$, showing a contradiction to 
the first part of~\eqref{pseudo-b}. 

For the first part of statement\eqref{pseudo-c}, assume to the contrary that there exists $\delta^*=\pbar(s_{j^*})$ for some $j^*\ne j$ and $j^*\in [t+1, \Delta-2]$ such that $P_{s_j}(\delta,\delta^*)\ne P_{s_{j^*}}(\delta,\delta^*)$. Then let $\varphi'=\varphi/P_{s_j}(\delta,\delta^*)$.
Note that $\varphi'$ is $(F,\varphi)$-stable, but $V(S)$ is not $\varphi'$-elementary, 
showing a contradiction to the  definition of a pseudo-multifan. 
For the second part of \eqref{pseudo-c}, 
assume that $r\not\in P_{s_j}(\delta,\delta^*)$
and the $(\delta,\delta^*)$-chain  containing $r$
is a path $Q$. By statement~\eqref{pseudo-a}, we let $s_{\ell_1}, \ldots, s_{\ell_k}$ be a rotation with $\varphi(rs_{\ell_1})=\pbar(s_{\ell_k})=\delta^*$ (note $s_{\ell_k}=s_{j^*}$). Note that the path $Q$ contains $rs_{j+1}$ and $rs_{\ell_1}$ since $\varphi(rs_{j+1})=\delta$ and $\varphi(rs_{\ell_1})=\delta^*$. So $Q-rs_{j+1}-rs_{\ell_1}$ consists of two disjoint paths, say $Q_j$ and $Q_{j^*}$, which contain $s_{j+1}$ and $s_{\ell_1}$ respectively. Let $\varphi'$ be obtained from $\varphi$ by shift from $s_{j}$ to $s_{l}$
and from $s_{\ell_1}$ to $s_{\ell_k}$ (only shift once if they are the same sequence up to permutation). Then $P_{s_{j+1}}(\delta,\delta^*,\varphi')=Q_j$. Let $\varphi^*=\varphi'/P_{s_{j+1}}(\delta, \delta^*,\varphi')$. We see that $\varphi^*$ is $(F,\varphi)$-stable, but $\delta^* \in \pbar^*(s_{j+1})\cap \pbar^*(s_{\ell_1})$, 
 giving a contradiction to the  definition of a pseudo-multifan.  \qed

\section{Proof of Theorems~\ref{Lem:2-non-adj1*} and~\ref{Lem:2-non-adj2}}

\subsection{Fundamental lemmas}
\begin{LEM}\label{Lemma:extended multifan}
Let  $(G,rs_1,\varphi)$ be a coloring-triple, 
$F:=F_\varphi(r,s_1:s_\alpha:s_\beta)$ be a typical multifan, and  $L:=(F,ru,u,ux,x)$ be a lollipop centered at $r$ such that  $\varphi(ru)=\alpha+1$ and $\pbar(x)=\alpha+1$.  Then  
	\begin{enumerate}[(a)]
		\item  $\varphi(ux)\ne 1$ and  $ux\in P_r(1,\varphi(ux))$.  \label{Evizingfan-a}
		
		If $\varphi(ux)=\tau$ 
		is a 2-inducing color with respect to $\varphi$ and $F$, then the following  holds. 
		\item Let $P_x(1,\tau)$ be the $(1,\tau)$-chain starting at $x$ in $G-rs_1-ux$. Then $P_x(1,\tau)$ ends at $r$.  \label{Evizingfan-b}
		\item For any 2-inducing color $\delta$  of $F$ with $\tau\prec \delta$, 
		we have $r\in P_{s_1}(\delta,\Delta)=P_{s_{\delta-1}}(\delta,\Delta)$. \label{Evizingfan-c}
		\item  For any $\Delta$-inducing color $\delta$ of $F$, we have $r\in P_{s_{\delta-1}}(\delta, \alpha+1 )=P_{s_{\alpha}}(\delta, \alpha+1)$, where $s_{\Delta-1}=s_1$ if $\delta=\Delta$.
		 \label{Evizingfan-d}
		 \item For any 2-inducing color $\delta$  of $F$ with $\delta\prec \tau$,  we have 
		 $r\in P_{s_\alpha}(\delta, \alpha+1)=P_{s_{\delta-1}}(\delta, \alpha+1)$. \label{Evizingfan-e}
	\end{enumerate}
\end{LEM}

\begin{proof}
	The assertion $\varphi(ux)\ne 1$ is clear. As otherwise, 
	$P_r(1,\alpha+1,\varphi)=rux$, contradicting Lemma~\ref{thm:vizing-fan1}~\eqref{thm:vizing-fan1b} that 
 $r$ and $s_\alpha$ are $(1,\alpha+1)$-linked  with respect to $\varphi$. Let $\varphi(ux)=\tau$ and assume $ux\notin P_r(1,\tau)$. 
 By the first part of~\eqref{Evizingfan-a}, $\tau \ne 1$.  
	Let $Q$ be the $(1,\tau)$-chain containing $u$.  Then $Q$ does not start or end at any vertex of $F$, since 
	$r$ and $\varphi_F^{-1}(\tau)$ are $(1,\tau)$-linked with respect to $\varphi$ by 
	Lemma~\ref{thm:vizing-fan1}~\eqref{thm:vizing-fan1b} if $\tau \in \pbar(F)$. 
	Therefore $\varphi'=\varphi/Q$  is $(F,\varphi)$-stable. Consequently,  $r$ and $s_\alpha$ are still $(1,\alpha+1)$-linked  with respect to $\varphi'$.   However $P_r(1,\alpha+1,\varphi')=rux$ ends at $x$,  giving a contradiction, proving~\eqref{Evizingfan-a}.

	For statement~\eqref{Evizingfan-b},  by~\eqref{Evizingfan-a},  it follows that $ux\in P_r(1,\tau)$. 
	Thus $P_x(1,\tau)$ is a subpath of $P_r(1,\tau)=P_{s_{\tau-1}}(1,\tau)$. 
	So $P_x(1,\tau)$   ends at either $r$  or $s_{\tau-1}$. 
	Assume to the contrary that $P_{x}(1,\tau)$ ends at $s_{\tau-1}$. Then $P_r(1,\tau)$ meets $u$ before $x$. Since $\varphi(rs_\tau)=\tau$,  it follows that $P_{[s_\tau,u]}(1,\tau)$ does not contain any edge from the lollipop $L$. Hence we can do the following operations, which are also demonstrated in Figure~\ref{f11}.  
	\[
\begin{bmatrix}
		s_\tau:s_\alpha & P_{[s_\tau,u]}(1,\tau)  & ru & ux\\
		\text{shift} & 1/\tau &  \alpha+1 \rightarrow 1 & \tau \rightarrow \alpha+1
	\end{bmatrix}.
	\]
\begin{figure}[!htb]
	\begin{center}
		\begin{tikzpicture}[scale=0.8]
			
{\tikzstyle{every node}=[draw ,circle,fill=white, minimum size=0.5cm,
	inner sep=0pt]
	\draw[black,thick](-1,0) node[label={left: }] (r)  {$r$};
	\draw[black,thick](-3,0) node[label={above: $s_1$}] (s1) {};
	\draw[black,thick](-3.5,2.5) node[label={above: $s_2$}] (s2) {};
	\draw[black,thick](-2,3) node[label={above: $s_{\tau-1}$}] (sa) {};
	\draw[black,thick](-0.5,3) node[label={above: $s_{\tau}$}] (sa1) {};
	\draw[black,thick](1.5,2.5) node[label={above: $s_\alpha$}] (sb) {};
	\draw[black,thick](-1,-1.5) node[label={left: }] (u)  {$u$};
	\draw[black,thick](-1,-3) node[label={left: }] (x)  {$x$};
}
\path[draw,thick,black, dashed]
(r) edge node[name=la,above,pos=0.5] {\color{black}} (s1);

\path[draw,thick,black]
(r) edge node[name=la,above,pos=0.5] {\color{black}$2$} (s2)
(r) edge node[name=la,above,pos=0.5] {\color{black}\quad \tiny \,\,$\tau-1$} (sa)
(r) edge node[name=la,right,pos=0.8] {\color{black} \tiny $\tau \color{red}{\rightarrow \tau+1}$} (sa1)
(r) edge node[name=la,right,pos=0.5] {\color{black}$\alpha \color{red}{\rightarrow \alpha+1}$} (sb)
(r) edge node[name=la,left,pos=0.5] {\color{black} $\alpha+1$} (u)
(u) edge node[name=la,left,pos=0.5] {\color{black} $\tau$} (x);

{\tikzstyle{every node}=[draw ,circle,fill=black, minimum size=0.05cm,
	inner sep=0pt]
	\draw(-3.05,2.7) node (f1)  {};
	\draw(-2.65,2.9) node (f1)  {};
	\draw(-2.4,3) node (f1)  {};
	\draw(0,3) node (f1)  {};
	\draw(0.5,2.9) node (f1)  {};
	\draw(1,2.7) node (f1)  {};
} 
\draw(2,0) node (f1)  { $\Rightarrow$};

\begin{scope}[shift={(6,0)}]
{\tikzstyle{every node}=[draw ,circle,fill=white, minimum size=0.5cm,
	inner sep=0pt]
	\draw[black,thick](-1,0) node[label={left: }] (r)  {$r$};
	\draw[black,thick](-3,0) node[label={above: $s_1$}] (s1) {};
	\draw[black,thick](-3.5,2.5) node[label={above: $s_2$}] (s2) {};
	\draw[black,thick](-2,3) node[label={above: $s_{\tau-1}$}] (sa) {};
	\draw[black,thick](-0.5,3) node[label={above: $s_{\tau}$}] (sa1) {};
	\draw[black,thick](1.5,2.5) node[label={above: $s_\alpha$}] (sb) {};
	\draw[black,thick](-1,-1.5) node[label={left: }] (u)  {$u$};
	\draw[black,thick](-1,-3) node[label={left: }] (x)  {$x$};
}
\path[draw,thick,black, dashed]
(r) edge node[name=la,above,pos=0.5] {\color{black}} (s1);

\path[draw,thick,black]
(r) edge node[name=la,above,pos=0.5] {\color{black}$2$} (s2)
(r) edge node[name=la,above,pos=0.5] {\color{black}\quad \tiny \,\,$\tau-1$} (sa)
(r) edge node[name=la,right,pos=0.8] {\color{black} \tiny $\tau \color{red}{\rightarrow \tau+1}$} (sa1)
(r) edge node[name=la,right,pos=0.5] {\color{black}$\alpha \color{red}{\rightarrow \alpha+1}$} (sb)
(r) edge node[name=la,left,pos=0.5] {\color{black} $\alpha+1$} (u)
(u) edge node[name=la,left,pos=0.5] {\color{black} $\tau$} (x);
\draw[thick, red] plot [smooth, tension=2] coordinates { (-0.24,3.1) (2,2.4) (-0.7,-1.5)};
\draw(0.4,3.8) node (f1)  { $1 \color{red}{\rightarrow \tau}$};
\draw(0.3,-1.5) node (f1)  { $1 \color{red}{\rightarrow \tau}$};

\draw(2,0) node (f1)  { $\Rightarrow$};

{\tikzstyle{every node}=[draw ,circle,fill=black, minimum size=0.05cm,
	inner sep=0pt]
	\draw(-3.05,2.7) node (f1)  {};
	\draw(-2.65,2.9) node (f1)  {};
	\draw(-2.4,3) node (f1)  {};
	\draw(0,3) node (f1)  {};
	\draw(0.5,2.9) node (f1)  {};
	\draw(1,2.7) node (f1)  {};
} 
\draw(2,0) node (f1)  { $\Rightarrow$};
\end{scope}	

\begin{scope}[shift={(12,0)}]
{\tikzstyle{every node}=[draw ,circle,fill=white, minimum size=0.5cm,
	inner sep=0pt]
	\draw[black,thick](-1,0) node[label={left: }] (r)  {$r$};
	\draw[black,thick](-3,0) node[label={above: $s_1$}] (s1) {};
	\draw[black,thick](-3.5,2.5) node[label={above: $s_2$}] (s2) {};
	\draw[black,thick](-2,3) node[label={above: $s_{\tau-1}$}] (sa) {};
	\draw[black,thick](-0.5,3) node[label={above: $s_{\tau}$}] (sa1) {};
	\draw[black,thick](1.5,2.5) node[label={above: $s_\alpha$}] (sb) {};
	\draw[black,thick](-1,-1.5) node[label={left: }] (u)  {$u$};
	\draw[black,thick](-1,-3) node[label={left: }] (x)  {$x$};
}
\path[draw,thick,black, dashed]
(r) edge node[name=la,above,pos=0.5] {\color{black}} (s1);

\path[draw,thick,black]
(r) edge node[name=la,above,pos=0.5] {\color{black}$2$} (s2)
(r) edge node[name=la,above,pos=0.5] {\color{black}\quad \tiny \,\,$\tau-1$} (sa)
(r) edge node[name=la,right,pos=0.8] {\color{black} \tiny $\tau \color{red}{\rightarrow \tau+1}$} (sa1)
(r) edge node[name=la,right,pos=0.5] {\color{black}$\alpha \color{red}{\rightarrow \alpha+1}$} (sb)
(r) edge node[name=la,left,pos=0.5] {\color{black} $\alpha+1 \color{red}{\rightarrow 1}$} (u)
(u) edge node[name=la,left,pos=0.5] {\color{black} $\tau \color{red}{\rightarrow \alpha+1}$} (x);
\draw[thick, red] plot [smooth, tension=2] coordinates { (-0.24,3.1) (2,2.4) (-0.7,-1.5)};
\draw(0.4,3.8) node (f1)  { $1 \color{red}{\rightarrow \tau}$};
\draw(0.3,-1.5) node (f1)  { $1 \color{red}{\rightarrow \tau}$};


{\tikzstyle{every node}=[draw ,circle,fill=black, minimum size=0.05cm,
	inner sep=0pt]
	\draw(-3.05,2.7) node (f1)  {};
	\draw(-2.65,2.9) node (f1)  {};
	\draw(-2.4,3) node (f1)  {};
	\draw(0,3) node (f1)  {};
	\draw(0.5,2.9) node (f1)  {};
	\draw(1,2.7) node (f1)  {};
} 
\end{scope}
	\end{tikzpicture}
	\caption{Coloring operations in the proof of Lemma~\ref{Lemma:extended multifan} (b).}
	\label{f11}
	\end{center}
\end{figure}
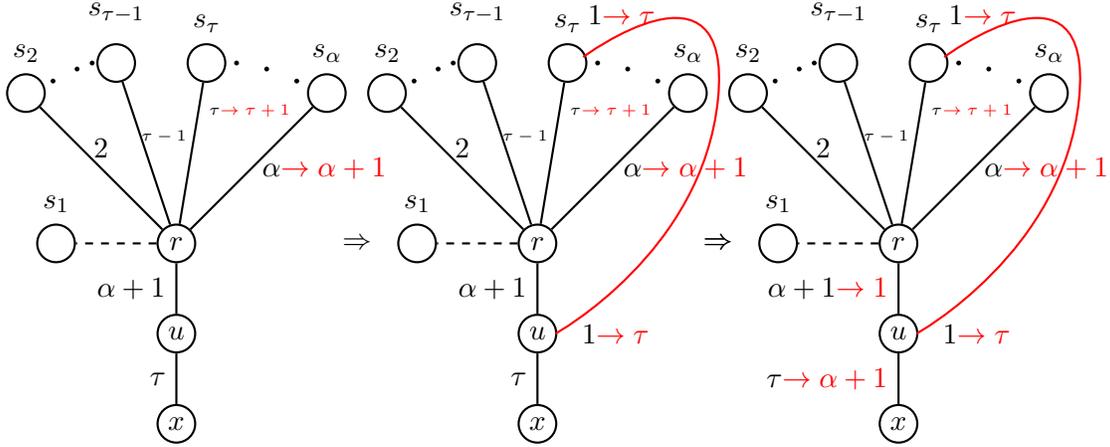

Clearly $(r,rs_1,s_1,\ldots,s_{\tau-1})$ is still a multifan under the new coloring, but $\tau$ is missing at both $r$ and $s_{\tau-1}$, showing a contradiction to Lemma~\ref{thm:vizing-fan1} \eqref{thm:vizing-fan1a}.	

Before proving the remaining statements, we introduce a new coloring $\varphi^*$ established on statement~\eqref{Evizingfan-b}. Let $\varphi^*$ be the coloring obtained from $\varphi$ by doing the following operations:
	\[
\begin{bmatrix}
		P_{x}(1,\tau)  & ru &ux\\
	1/\tau & \alpha+1 \rightarrow \tau & \tau \rightarrow \alpha+1
	\end{bmatrix},
	\]
where $P_{x}(1,\tau)$ is the $(1,\tau)$-chain starting at $x$ in $G-rs_1-ux$, which ends at $r$ by statement~\eqref{Evizingfan-b}. Let $E_{ch}=E(P_{x}(1,\tau))\cup \{ux,ur\}$. Clearly $\varphi^*$ and $\varphi$ agree on all edges in $E(G)\setminus E_{ch}$. Note that $\pbar^*(r)=\alpha+1$ and $\pbar^*(s)=\pbar(s)$ for all $s\in V(F)\setminus\{r\}$, and both $(r,rs_1,s_1,rs_2,s_2\ldots,s_{\tau-1})$ and $(r,rs_1,s_1,rs_{\alpha+1},s_{\alpha+1},\ldots,s_{\beta})$ are multifans with respect to $\varphi^*$. These properties will be frequently used in the following proof.

%

%
%

Now for the statement~\eqref{Evizingfan-c},  we have 
 $P_{s_1}(\delta,\Delta,\varphi)=P_{s_{\delta-1}}(\delta,\Delta,\varphi)$ by Lemma~\ref{thm:vizing-fan2}~\eqref{thm:vizing-fan2-a}.
 Assume to the contrary that $r\not\in P_{s_1}(\delta,\Delta,\varphi)$. Since $\{\delta,\Delta\}\cap \{1,\tau\}=\emptyset$, $P_{s_1}(\delta,\Delta,\varphi)=P_{s_{\delta-1}}(\delta,\Delta,\varphi)$ does not contain any edge from $E_{ch}$. Thus $P_{s_1}(\delta,\Delta, \varphi^*)=P_{s_{\delta-1}}(\delta,\Delta,\varphi^*)=P_{s_1}(\delta,\Delta, \varphi)$, and so $r\notin P_{s_1}(\delta,\Delta, \varphi^*)$. Let $\varphi'=\varphi^*/P_{s_1}(\delta,\Delta, \varphi^*)$. Then $\delta\in \pbar'(s_1)$ and $(r,rs_1,s_1,rs_{\delta},s_{\delta},\ldots,s_{\alpha})$ is a multifan under $\varphi'$.  However $\alpha+1$ is missing at both $r$ and $s_{\alpha}$, giving a contradiction to Lemma~\ref{thm:vizing-fan1} \eqref{thm:vizing-fan1a}.


For  statement~\eqref{Evizingfan-d},  we have 
$P_{s_{\delta-1}}(\alpha+1, \delta,\varphi)=P_{s_{\alpha}}(\alpha+1, \delta,\varphi)$ by Lemma~\ref{thm:vizing-fan2} \eqref{thm:vizing-fan2-a}.
Assume to the contrary that $r\not\in P_{s_{\delta-1}}(\alpha+1, \delta,\varphi)$. Since $\{\alpha+1, \delta\}\cap \{1,\tau\}=\emptyset$, $P_{s_\alpha}(\alpha+1, \delta,\varphi)=P_{s_{\delta-1}}(\alpha+1, \delta,\varphi)$ does not contain any edge from $E_{ch}$. Thus $P_{s_\alpha}(\alpha+1, \delta,\varphi^*)=P_{s_{\delta-1}}(\alpha+1, \delta,\varphi^*)=P_{s_{\delta-1}}(\alpha+1, \delta,\varphi)$, and so $r\notin P_{s_{\delta-1}}(\alpha+1, \delta,\varphi^*)$. 
This gives a contradiction to Lemma~\ref{thm:vizing-fan1}~\eqref{thm:vizing-fan1b} that 
$r$ and $s_{\delta-1}$ are $(\alpha+1, \delta)$-linked with respect to $\varphi^*$, since 
$(r,rs_1,s_1,rs_{\alpha+1},s_{\alpha+1},\ldots,s_{\beta})$ is a multifan under $\varphi^*$.


For  statement~\eqref{Evizingfan-e},  if it fails then we have that either  $P_{s_\alpha}(\delta,\alpha+1,\varphi)\ne P_{s_{\delta-1}}(\delta,\alpha+1,\varphi)$ or 
$P_{s_\alpha}(\delta,\alpha+1,\varphi)=P_{s_{\delta-1}}(\delta,\alpha+1,\varphi)$ but 
$r\not\in P_{s_\alpha}(\delta,\alpha+1,\varphi)$. 
For the first case,  $r\in P_{s_\alpha}(\delta,\alpha+1,\varphi)$ by Lemma~\ref{thm:vizing-fan2}~\eqref{thm:vizing-fan2-b} and so $r\not\in P_{s_{\delta-1}}(\delta,\alpha+1,\varphi)$. 
Therefore  $r\not\in P_{s_{\delta-1}}(\delta,\alpha+1,\varphi)$ in both cases. 
Consequently, $E(P_{\delta-1}(\delta,\alpha+1,\varphi))\cap E_{ch}=\emptyset$.
Hence,  $P_{s_{\delta-1}}(\delta,\alpha+1,\varphi^*)= P_{s_{\delta-1}}(\delta,\alpha+1,\varphi)$
and $r\not\in P_{s_{\delta-1}}(\delta,\alpha+1,\varphi^*)$.  This gives a contradiction, since 
under $\varphi^*$, $(r,rs_1,s_1,rs_2,s_2\ldots,s_{\tau-1})$ is a multifan, and so $r$ and $s_{\delta-1}$ should be $(\delta,\alpha+1)$-linked by Lemma~\ref{thm:vizing-fan1}~\eqref{thm:vizing-fan1b}.  This finishes  the proof of statement~\eqref{Evizingfan-e}.  \end{proof}

Let  $(G,rs_1,\varphi)$ be a coloring-triple. 
For a color $\alpha \in [1,\Delta]$, a sequence of 
{\it Kempe  $(\alpha,*)$-changes}  is a sequence of  
Kempe changes that each involve the exchanging of the color $\alpha$
with another color from $[1,\Delta]$.  

	\begin{LEM}\label{Lemma:pseudo-fan0}
	Let  $(G,rs_1,\varphi)$ be a coloring-triple, 
	$F:=F_\varphi(r,s_1:s_\alpha:s_\beta)$ be a typical multifan, and  $L:=(F,ru,u,ux,x)$ be a lollipop centered at $r$ such that  $\varphi(ru)=\alpha+1$.   Then for   $w_1\in \{s_{\beta+1}, \ldots, s_{\Delta-2} \}$ with $\varphi(rw_1)=\tau_1\in [\beta+2, \Delta-1]$,  the following statements hold.
		\begin{enumerate}[(1)] 
			\item \label{only-fan}	
			If exists a vertex $w\in V(G)\setminus (V(F)\cup \{w_1\})$ such that $w\in P_r(1,\tau_1,\varphi')$
			for every {\bf $(F,\varphi)$-stable} $\varphi' \in \CC^{\Delta}(G-rs_1)$ with $\varphi'(ru)=\alpha+1$,   
					then there exists a sequence of distinct vertices $w_1, \ldots, w_t\in  \{s_{\beta+1}, \ldots, s_{\Delta-2} \}$ satisfying the following conditions: 
		\begin{enumerate}[(a)]
			\item  $\varphi(rw_{i+1})=\pbar(w_i)\in  [\beta+2, \Delta-1]$ for each $i\in [1,t-1]$;
						 \label{Lemma:pseudo-fan0-a}
			\item  $r$ and $w_i$  are $(1,\pbar(w_i))$-linked with respect to $\varphi$ for each $i\in[1,t]$; \label{Lemma:pseudo-fan0-b}
			\item $\pbar(w_t)=\tau_1$. \label{Lemma:pseudo-fan0-c}
				\end{enumerate}
			\item \label{fan-and-x}
			If  $\pbar(x)=\alpha+1$ and there exists a vertex $w\in V(G)\setminus (V(F)\cup \{w_1\})$ such that $w\in P_r(1,\tau_1,\varphi')$
for every  {\bf $(L,\varphi)$-stable} $\varphi' \in \CC^{\Delta}(G-rs_1)$  obtained from $\varphi$ 
through a sequence of Kempe  $(1,*)$-changes  not using  $r$
or $x$ as endvertices, 
then there exists a sequence of distinct vertices $w_1, \ldots, w_t\in  \{s_{\beta+1}, \ldots, s_{\Delta-2} \}$ satisfying the following conditions: 
	\begin{enumerate}[(a)]
	\item  $\varphi(rw_{i+1})=\pbar(rw_i)\in  [\beta+2, \Delta-1]$  for each $i\in [1,t-1]$;
	\label{Lemma:pseudo-fan0-a1}
	\item $r$ and $w_i$   are $(1,\pbar(w_i))$-linked with respect to $\varphi$ for each $i\in[1,t-1]$; \label{Lemma:pseudo-fan0-b1}
	\item $\pbar(w_t)=\tau_1$ or $\pbar(w_t)=\alpha+1$.  If $\pbar(w_t)=\tau_1$,
	then $w_t$ and $r$ are $(1,\tau_{1})$-linked with respect to $\varphi$. \label{Lemma:pseudo-fan0-c1}
\end{enumerate}
	\end{enumerate}
\end{LEM} 
\begin{proof}
	We show~\eqref{only-fan} and \eqref{fan-and-x} simultaneously.  Let $\pbar(w_1)=\tau_2$.
	For \eqref{fan-and-x}, we may assume that $\tau_2\ne \alpha+1$, as otherwise, $w_1$
	is the desired sequence for it.  Note that $\tau_2\ne 1$, since otherwise $w\notin P_r(1,\tau_1)=rw_1$, giving a contradiction.
	We claim that $w_1$ satisfies $(a)$ and $(b)$ in  \eqref{only-fan}. 
	If  $(a)$ fails, then $\tau_2\in [1,\Delta]\setminus[\beta+2, \Delta-1]=\pbar(F)$. In this case we know that $r$ and $\mathit{\pbar^{-1}_F(\tau_2)}$ are $(1,\tau_2)$-linked by Lemma~\ref{thm:vizing-fan1}~\eqref{thm:vizing-fan1b}; if $(b)$ fails, then $w_1$ and $r$ are $(1,\tau_2)$-unlinked. In both cases, we have $w_1$ and $r$ are $(1,\tau_2)$-unlinked. Now let $\varphi'=\varphi/P_{w_1}(1,\tau_2)$.  
	Clearly, $\varphi'$ is $(F,\varphi)$-stable. Since 
	$r$ and $s_\alpha$ are $(1,\alpha+1)$-linked 
	with respect to $\varphi$ by Lemma~\ref{thm:vizing-fan1}~\eqref{thm:vizing-fan1b}
	and $\varphi(ru)=\alpha+1$, 
	we have  $\varphi'(ru)=\varphi(ru)$.
	For ~\eqref{only-fan}, we already achieve a contradiction
	since $\pbar'(w_1)=1$ implies 
	$w\not\in P_r(1,\tau_1,\varphi')=rw_1$. 
	For~\eqref{fan-and-x},  since  $\tau_2\ne \alpha+1$,  $\pbar'(x)=\pbar(x)$. 
	By Lemma~\ref{Lemma:extended multifan}~\eqref{Evizingfan-a},
	the color $\varphi(ux)$ on $ux$ will keep unchanged under 
	any Kempe $(1,*)$-change not using $r$ or $x$ as endvertices. 
	Thus, $\varphi'$ is $(L,\varphi)$-stable. We again reach a contradiction 
	since  $w\notin P_r(1,\tau_1,\varphi')=rw_1$.

	Now $w_1$ is a sequence that satisfies $(a)$ and $(b)$ in~\eqref{only-fan} and so in~\eqref{fan-and-x}.  
	 Let $w_1, \ldots, w_k$ be a longest sequence of vertices from $ \{s_{\beta+1}, \ldots, s_{\Delta-2} \}$ that satisfies $(a)$ and $(b)$ in \eqref{only-fan}. 
	 	 Let $\pbar(w_i)=\tau_{i+1}$ for each $i\in [1,k]$. 
	 	 We are done if $\pbar(w_k)=\tau_1$. 
	 Thus  $\pbar(w_k)=\tau_{k+1} \ne \tau_1$. 
	 Since the sequence satisfies statement (1) (a),  we have $\tau_{k+1}\in [\beta+2, \Delta-1]$. Since  $w_i$ is $(1,\tau_{i+1})$-linked with $r$ for each $i\in [1,k]$, we know  $\tau_{k+1}\notin \{\tau_1,\tau_2,\ldots,\tau_{k}\}$. Thus, there exists $w_{k+1}\in N_{\Delta-1}(r)$ such that $\varphi(rw_{k+1})=\tau_{k+1}$.   Let $\pbar(w_{k+1})=\tau_{k+2}$.
	 By the maximality of the sequence $w_1,\ldots, w_k$, 
	 either  $\pbar(w_{k+1})\in \pbar(F)$ or $\pbar(w_k)\in [\beta+2, \Delta-1]$ and $w_{k+1}$ and $r$ are $(1,\tau_{k+2})$-unlinked.  
	 In both cases,  
	 $w_{k+1}$ and $r$ are $(1,\tau_{k+2})$-unlinked.  
	 Again,  for \eqref{fan-and-x},  we assume  $\pbar(w_{k+1}) \ne \alpha+1$. 
	 Thus, we assume that we are  proving \eqref{only-fan} and proving \eqref{fan-and-x}
	under the assumption that $\pbar(w_{k+1})\ne \alpha+1$. 
		In both cases,  let $\varphi_0=\varphi$, we 
	do a sequence of Kempe changes around $r$ from $w_{k+1}$ to $w_1$ as below:
	$$\varphi_j=\varphi_{j-1}/P_{w_{k+1-(j-1)}}(1,\tau_{k+2-(j-1)}, \varphi_{j-1}) \quad \text{for each $j\in [1,k+1]$}.$$
	Note that $$P_{r}(1,\tau_{k+1-(j-1)},\varphi_j)=rw_{k+1-(j-1)} \quad \text{for each $j\in [1,k+1]$}. $$
	
	 Clearly,  $\varphi_{k+1}$ is obtained from $\varphi$ 
	through a sequence of   Kempe $(1,*)$-changes  not using $r$ as endvertices. For \eqref{only-fan},  
	$\varphi_{k+1}$ is $(F,\varphi)$-stable such that $\varphi_{k+1}(ru)=\varphi(ru)$. 
	For the case of proving~\eqref{fan-and-x},   as we assumed  $\pbar(w_{k+1})\ne \alpha+1$,  each of the   Kempe $(1,*)$-changes  from the sequence did not use $x$ as an endvertex. Thus we still have  
  $\pbar_{k+1}(x)=\pbar(x)=\alpha+1$.   As a consequence, 
		by Lemma~\ref{Lemma:extended multifan}~\eqref{Evizingfan-a},  	the color $\varphi(ux)$ on $ux$ is kept unchanged.  
	Therefore $\varphi_{k+1}$
	is $(L,\varphi)$-stable in the case of proving~\eqref{fan-and-x}. 
	However, in both cases, 
	$w\not\in P_r(1,\tau_1,\varphi_{k+1})=rw_1$.  
	This gives a contradiction to the assumptions  in~\eqref{only-fan} and \eqref{fan-and-x}. 
\end{proof}
	
	 By the definition,  $w_1, \ldots, w_t$ in Lemma~\ref{Lemma:pseudo-fan0}~\eqref{only-fan} and in the case of  Lemma~\ref{Lemma:pseudo-fan0}~\eqref{fan-and-x} when $\pbar(w_t)=\tau_1$ form a rotation
	 with the additional property that $\pbar(w_i) \in [\beta+2,\Delta-1]$ and $r$ and $w_i$ are $(1,\pbar(w_i))$-linked for each $i\in[1,t]$. We call such a rotation a \emph{stable rotation}. 
	 In the case of  Lemma~\ref{Lemma:pseudo-fan0}~\eqref{fan-and-x} when $\pbar(w_t)=\alpha+1$,
	 we call $w_1, \ldots, w_t$ a \emph{near stable rotation}. 

\subsection{Proof of Theorem~\ref{Lem:2-non-adj1*}}

\begin{THM3}
	Let  $(G,rs_1,\varphi)$ be a coloring-triple, 
$F:=F_\varphi(r,s_1:s_\alpha:s_\beta)$ be a typical multifan, and  $L:=(F,ru,u,ux,x)$ be a lollipop centered at $r$.  If $\varphi(ru)=\alpha+1$, $\pbar(x)=\alpha+1$, 
and $\varphi(ux)=\Delta$, 
then  the following two statements hold.
\begin{enumerate}[(1)]
\item If $u\sim s_1$, then $\varphi(us_1)$ is a $\Delta$-inducing color of $F$. 
\item If $u\sim s_\alpha$, then $\varphi(us_\alpha)$ is a $\Delta$-inducing color of $F$.
\end{enumerate}
\end{THM3}

\begin{proof}
	Assume to the contrary that the statements fail. We 
	naturally have two cases. 
	
	\smallskip 
	
	{\noindent \bf Case 1: $u\sim s_1$ and $\varphi(us_1)$ is not a $\Delta$-inducing color}.  
					Let $\varphi(us_1)=\tau$. 	Note that $\tau\ne 2, \alpha+1, \Delta$. We first show that $us_1$ can not be 1 under any $(L,\varphi)$-stable coloring.
	\smallskip


	\begin{CLA}\label{not1}
For every  $(L,\varphi)$-stable $\varphi^*\in \CC^\Delta(G-rs_1)$, it holds that $\varphi^*(us_1)\ne 1$.
Furthermore, if 
 $\varphi^*(us_1)=\varphi(us_1)=\tau$,   then $us_1\in P_r(1,\tau, \varphi^*)$. 
	\end{CLA}
\proof[Proof of Claim~\ref{not1}] 
Suppose instead  that $\varphi^*(us_{1})=1$ for the first part,  
and $us_{1}\notin P_r(1,\tau, \varphi^*)$  for the second part.  Let $\varphi'=\varphi^*$ in the former case and let $\varphi'=\varphi^*/Q$ in the latter case,  where $Q$ is the $(1,\tau)$-chain containing $us_1$. In both cases $\varphi'(us_1)=1$ and so $us_1\in P_r(1,\alpha+1,\varphi')$. Clearly $\varphi'$ is $(F,\varphi)$-stable. Since $P_r(1,\alpha+1,\varphi')= P_{s_\alpha}(1,\alpha+1,\varphi')$ by Lemma~\ref{thm:vizing-fan1}~\eqref{thm:vizing-fan1b}, $P_x(1,\alpha+1,\varphi')$ does not contain $r$. Thus $\varphi''=\varphi'/P_x(1,\alpha+1,\varphi')$ is $(F,\varphi^*)$-stable and $\varphi''(us_1)=\varphi'(us_1)=1$.  However,  $P_{s_1}(1,\Delta,\varphi'')=s_1ux$, contradicting Lemma~\ref{thm:vizing-fan1}~\eqref{thm:vizing-fan1b} that $s_1$ and $r$ are $(1,\Delta)$-linked with respect to $\varphi''$.  
\qed 

By Claim~\ref{not1}, we have either $\tau\in \pbar(F)$ is $2$-inducing or 
$\tau \in [\beta+2,\Delta-1]$. 

\smallskip 
{\noindent \bf Subcase 1.1: $\tau\in \pbar(F)$ is $2$-inducing}.
\smallskip

By Claim~\ref{not1},  $us_1\in P_r(1,\tau)=P_{s_{\tau-1}}(1,\tau)$. Let $P_u(1,\tau)$ be the
$(1,\tau)$-chain starting at $u$ in $G-rs_1-us_1$, which is a subpath of $P_r(1,\tau)$. Then $P_u(1,\tau)$ ends at  either $r$ or $s_{\tau-1}$. Consequently if we shift from $s_\tau$ to $s_\alpha$, then $P_u(1,\tau)$ will end at  either $s_\tau$ or $s_{\tau-1}$. Thus we can do the following operations:
\[
\begin{bmatrix}
s_{\tau}:s_{\alpha} & P_{u}(1,\tau)  & us_1 & ux & ur\\
\text{shift} & 1/\tau & \tau \rightarrow \Delta & \Delta \rightarrow \alpha+1 & \alpha+1 \rightarrow 1
\end{bmatrix}.
\]

Denote the new coloring by $\varphi'$. Now $\pbar'(s_1)=\pbar'(r)=\tau$, we can color $rs_1$ by $\tau$ to obtain an edge $\Delta$-coloring of $G$, which contradicts the fact that $G$ is class 2.

\smallskip 
{\noindent \bf Subcase 1.2: $\tau \in [\beta+2,\Delta-1]$}.  
\smallskip

Let $w_1\in N_{\Delta-1}(r)$ such that $\varphi(rw_1)=\tau$. 
Since $us_1\in P_r(1,\tau, \varphi) $  and  the color on $us_{1}$ 
is not 1 under  every $(L,\varphi)$-stable coloring, by Claim~\ref{not1},  
for every  $(L,\varphi)$-stable $\varphi' \in \CC^{\Delta}(G-rs_1)$  obtained from $\varphi$ 
through a sequence of  Kempe $(1,*)$-changes not using $r$ or $x$ as endvertices  it holds that 
$\varphi'(us_1)=\varphi(us_1)=\tau$.   Therefore, 
$u\in P_r(1,\tau, \varphi')$  by Claim~\ref{not1} again.  (Note that every Kempe $(1,*)$-change $\varphi^*$ not using $r$ or $x$ as an endvertex is $(L,\varphi)$-stable. 
Since $r$ and $s_i$ are $(1,\delta)$-linked for every $\delta\in \pbar(s_i)$ with $i\in [1,\beta]$,  the Kempe $(1,*)$-chain not using $r$ or $x$ as an endvertex implies that it did not use any vertex from  $V(F)\cup \{x\}$ as an endvertex. Thus $\varphi^*(ru)=\varphi(ru)=\pbar^*(x)=\pbar(x)=\alpha+1$. 
By Lemma~\ref{Lemma:extended multifan}~\eqref{Evizingfan-a}, it further implies that 
$\varphi^*(ux)=\varphi(ux)$. Thus $\varphi^*$ is $(L,\varphi)$-stable.)
Applying  Lemma~\ref{Lemma:pseudo-fan0} on $L$ with $u$ playing the role of $w$, 
we find a sequence of distinct vertices $w_1, \ldots, w_t\in  \{s_{\beta+1}, \ldots, s_{\Delta-2} \}$ that forms either a stable rotation or a near stable rotation.  

 Assume first that $w_1, \ldots, w_t$ is a stable rotation.
In this case,  $t\ge 2$ and $r$ and $w_i$ are $(1,\pbar(w_i))$-linked for each $i\in[1,t]$. By Claim~\ref{not1},  $us_1\in P_r(1,\tau)=P_{w_t}(1,\tau)$. 
If $P_{w_t}(1,\tau)$ meets $u$ before $s_1$, then  
we do the following operations:
\[
\begin{bmatrix}
P_{[w_t,u]}(1,\tau)  &  us_1 & ux & ur\\
1/\tau & \tau \rightarrow \Delta & \Delta \rightarrow \alpha+1 & \alpha+1 \rightarrow 1
\end{bmatrix}.
\]
The new coloring is $(F,\varphi)$-stable, but $\alpha+1$ is missing at both $r$ and $s_\alpha$, giving a contradiction to Lemma~\ref{thm:vizing-fan1}~\eqref{thm:vizing-fan1a}.  {Thus $P_{r}(1,\tau)$ meets $u$ before $s_1$}.  Then shift from $w_1$
to $w_t$  gives back to the previous case with $w_1$ playing the role of $w_t$. 


Assume then that $w_1, \ldots, w_t$ is a near stable rotation, i.e., $\pbar(w_t)=\alpha+1$.
Note that $t$ could be 1 in this case. 
By Claim~\ref{not1},  $us_1\in P_r(1,\tau)=P_{z}(1,\tau)$, for some vertex
$z\in V(G)\setminus (V(F)\cup \{x,w_1, \ldots, w_t\})$. 
{Assume first that $w_t\ne x$}. 
{If $P_z(1,\tau)$ meets $u$ before $s_1$},  we do the following operations: 
\[
\begin{bmatrix}
P_{[z,u]}(1,\tau)  &  us_1 & ux & ur\\
1/\tau & \tau \rightarrow \Delta & \Delta \rightarrow \alpha+1 & \alpha+1 \rightarrow 1
\end{bmatrix}.
\]
The new coloring is $(F,\varphi)$-stable, but $\alpha+1$ is missing at both $r$ and $s_\alpha$, giving a contradiction to Lemma~\ref{thm:vizing-fan1}~\eqref{thm:vizing-fan1a}.
{If $P_z(1,\tau)$ meets $s_1$ before $u$},  we do the following operations:
\[
\begin{bmatrix}
P_{[z,s_1]}(1,\tau) & w_1:w_t  &  us_1 & ux & ur\\
1/\tau & \text{shift} & \tau \rightarrow \Delta & \Delta \rightarrow \alpha+1 & \alpha+1 \rightarrow \tau
\end{bmatrix}.
\]
The new coloring is $(F,\varphi)$-stable,  but $1$ is missing at both $r$ and $s_1$, giving a contradiction to Lemma~\ref{thm:vizing-fan1}~\eqref{thm:vizing-fan1a}.

%

{Assume now  that $w_t= x$}. 
We first consider the case when $t\ge 2$. Note that $P_r(1,\alpha+1)=P_{s_\alpha}(1,\alpha+1)$ and so $r\notin P_x(1,\alpha+1)$.  Let $\varphi_1=\varphi/P_x(1,\alpha+1)$. Then $P_r(1,\tau_{t},\varphi_1)=rx$, where $\tau_t:=\varphi(rw_t)$.  We next let  $\varphi_2=\varphi_1/P_{w_{t-1}}(1,\tau_t,\varphi_1)$. Then $\varphi_2$ is $(F,\varphi)$-stable and $\pbar_2(x)=\pbar_2(w_{t-1})=1$. Now doing a $(1,\alpha+1)$-swap at both $x$ and $w_{t-1}$ 
 gives back to the previous case when $\pbar(w_t)=\alpha+1$ and $w_t\ne x$
with $w_{t-1}$ in place of $w_t$.

Thus we assume that $t=1$. Let $\varphi_1=\varphi/P_x(1,\alpha+1)$. Then $P_r(1,\tau,\varphi_1)=rx$. We next let  $\varphi_2=\varphi_1/Q$,  where $Q$ is the $(1,\tau)$-chain containing $us_1$ under $\varphi_1$. Then $\varphi_2$ is $(F,\varphi)$-stable, but  $P_{s_1}(1,\Delta,\varphi_2)=s_1ux$ ends at $x$, giving a contradiction to Lemma~\ref{thm:vizing-fan1}~\eqref{thm:vizing-fan1b}.
%
	\smallskip 

{\noindent \bf Case 2: $u\sim s_\alpha$ and $\varphi(rs_\alpha)$ is not a $\Delta$-inducing color}.  
\smallskip

Let $\varphi(us_\alpha)=\tau$. 	Note that $\tau\ne \alpha, \alpha+1, \Delta$. By Lemma~\ref{thm:vizing-fan2}~\eqref{thm:vizing-fan2-a}, $P_{s_1}(\alpha+1,\Delta)=P_{s_\alpha}(\alpha+1,\Delta)$. Since $r\in P_x(\alpha+1,\Delta)$, we have $r\notin P_{s_1}(\alpha+1,\Delta)$. Now let $\varphi'=\varphi/P_{s_1}(\alpha+1,\Delta)$ and let $\varphi^*$ be obtained from $\varphi'$ by uncoloring $rs_\alpha$, shift from $s_2$ to $s_{\alpha-1}$ and coloring $rs_1$ by 2. Then $F^*=(r,rs_{\alpha},s_{\alpha},rs_{\alpha-1},s_{\alpha-1},\ldots,s_1,$ $rs_{\alpha+1},s_{\alpha+1},\ldots,s_\beta)$ is a  multifan centered at $r$ with respect to $rs_\alpha$ and $\varphi^*$, where $\pbar^*(s_\alpha)=\{\alpha,\Delta\}$, $\varphi^*(ru)=\pbar^*(x)=\alpha+1$ is the last $\alpha$-inducing color, $\varphi^*(ux)=\Delta$, and $u\sim s_\alpha$. 
Since the $\Delta$-sequence of $F^*$ agrees with that of $F$, $\tau$ is still not $\Delta$-inducing with respect to $F^*$ and $\varphi^*$.  Furthermore, we can assume $F^*$ is typical by renaming colors in $ [2,\alpha]$ and vertices in $ \{s_1,\ldots, s_\alpha\}$. Thus the current case is reduced 
to {Case 1},  finishing  the proof of Theorem~\ref{Lem:2-non-adj1*}.  \end{proof}

\subsection{Proof of Theorem~\ref{Lem:2-non-adj2}}

\begin{THM4}
	Let  $(G,rs_1,\varphi)$ be a coloring-triple, 
	$F:=F_\varphi(r,s_1:s_\alpha)$ be a typical 2-inducing  multifan, and  $L:=(F,ru,u,ux,x)$ be a lollipop centered at $r$.  If $\varphi(ru)=\alpha+1$, $\pbar(x)=\alpha+1$, 
	and $\varphi(ux)=\mu\in \pbar(F)$ is a 2-inducing color of $F$, 
	then  $u\not\sim s_{\mu-1}$ and $u\not\sim s_\mu$.  
\end{THM4}

\begin{proof}
	Assume to the contrary that $u\sim s_{\mu-1}$ or $u\sim s_\mu$. We consider two cases below. 
	
	\smallskip 
	
	{\noindent \bf Case 1: $u\sim s_{\mu-1}$}. 
	 	Let $\varphi(us_{\mu-1})=\tau$. 	Note that $\tau\ne \mu-1, \mu,  \alpha+1$. 
	\smallskip

	\begin{CLA}\label{not1_3}
		For every  $(L,\varphi)$-stable $\varphi^*\in \CC^\Delta(G-rs_1)$, it holds that	$\varphi^*(us_{\mu-1})\ne 1$.  Furthermore, if  
		$\varphi^*(us_{\mu-1})=\varphi(us_{\mu-1})=\tau$,  then  $us_{\mu-1}\in P_r(1,\tau, \varphi^*)$.  
	\end{CLA}
	\proof[Proof of Claim~\ref{not1_3}] 
	Suppose instead  that $\varphi^*(us_{\mu-1})=1$ for the first part,  
	and $us_{\mu-1}\notin P_r(1,\tau, \varphi^*)$  for the second part. 
	Let $\varphi'=\varphi^*$ in the former case and let $\varphi'=\varphi^*/Q$ in the latter case, where $Q$ is the $(1,\tau)$-chain containing $us_{\mu-1}$. Clearly $\varphi'$ is $(F,\varphi)$-stable. Since $P_r(1,\alpha+1,\varphi')= P_{s_\alpha}(1,\alpha+1,\varphi')$ by Lemma~\ref{thm:vizing-fan1}~\eqref{thm:vizing-fan1b}, $P_x(1,\alpha+1,\varphi')$ does not contain $r$. Thus $\varphi''=\varphi'/P_x(1,\alpha+1,\varphi')$ is $(F,\varphi^*)$-stable and $\varphi''(us_{\mu-1})=\varphi'(us_{\mu-1})=1$.  However  $P_{s_{\mu-1}}(1,\mu,\varphi'')=s_{\mu-1}ux$, contradicting Lemma~\ref{thm:vizing-fan1}~\eqref{thm:vizing-fan1b} that $s_{\mu-1}$ and $r$ are $(1,\mu)$-linked.
	\qed 
	
By Claim~\ref{not1_3}, we have either $\tau \in \pbar(F)\setminus \{1,\mu-1,\mu,\alpha+1\}$
or $\tau \in [\alpha+2,\Delta-1]$.  
	
	\smallskip 
{\noindent \bf \setword{Subcase 1.1}{Subcase 1.1}: $\tau\in \pbar(F)$}.  
\smallskip 

Assume first that $\tau \prec \mu$.  
	By Lemma~\ref{Lemma:extended multifan} \eqref{Evizingfan-e},  $r\in P_{s_\alpha}(\tau,\alpha+1)=P_{s_{\tau-1}}(\tau,\alpha+1)$.  
	Let $\varphi'=\varphi/P_x(\tau,\alpha+1)$. 
	Then $P_x(\tau,\mu,\varphi')=xus_{\mu-1}$. Swapping colors 
	along $P_x(\tau,\mu,\varphi')=xus_{\mu-1}$ to get a new coloring $\varphi''$. Then both $s_{\tau-1}$
	and $s_{\mu-1}$ miss $\tau$ with respect to $\varphi''$, giving a contradiction
	to Lemma~\ref{thm:vizing-fan1}~\eqref{thm:vizing-fan1a} that $V(F_{\varphi''}(r,s_1: s_{\mu-1}))$ is $\varphi''$-elementary.
	
	Assume then that  $\tau=\Delta$.   
	By Lemma~\ref{Lemma:extended multifan} \eqref{Evizingfan-d},  $r\in P_{s_\alpha}(\alpha+1,\Delta)=P_{s_1}(\alpha+1,\Delta)$.  
	Let $\varphi'=\varphi/P_x(\alpha+1,\Delta)$. 
	Then $P_x(\Delta,\mu,\varphi')=xus_{\mu-1}$. Swapping colors 
	along $P_x(\Delta,\mu,\varphi')=xus_{\mu-1}$ to get a new coloring $\varphi''$. Then  both $s_1$
	and $s_{\mu-1}$ miss $\Delta$ with respect to $\varphi''$, giving a contradiction
	to Lemma~\ref{thm:vizing-fan1}~\eqref{thm:vizing-fan1a} that $V(F_{\varphi''}(r,s_1: s_{\mu-1}))$ is $\varphi''$-elementary. 
	
	Assume lastly that $\mu \prec \tau \prec \alpha+1$.  
	By Claim~\ref{not1_3}, $us_{\mu-1}\in P_r(1,\tau)=P_{s_{\tau-1}}(1,\tau)$. Let $P_u(1,\tau)$ be the  $(1,\tau)$-chain starting at $u$ in $G-rs_1-us_{\mu-1}$, 
	which is a subpath of $P_r(1,\tau)$. Then $P_u(1,\tau)$ ends at  either $r$ or $s_{\tau-1}$. Consequently if we shift from $s_\tau$ to $s_\alpha$, then $P_u(1,\tau)$ will end at either $s_\tau$ or $s_{\tau-1}$. Thus we can do the following operations:
	\[
	\begin{bmatrix}
	s_{\tau}:s_{\alpha} & P_{u}(1,\tau)  & us_{\mu-1} & ux & ur\\
	\text{shift} & 1/\tau & \tau \rightarrow \mu & \mu \rightarrow \alpha+1 & \alpha+1 \rightarrow 1
	\end{bmatrix}.
	\]
	
	Denote the new coloring by $\varphi'$. Now $(r,rs_1,s_1,\ldots,s_{\mu-1})$ is a multifan, but $\pbar'(s_{\mu-1})=\pbar'(r)=\tau$, giving a contradiction to Lemma~\ref{thm:vizing-fan1}~\eqref{thm:vizing-fan1a}.
	
\smallskip 
	{ \bf\noindent Subcase 1.2:  $\tau\in [\alpha+2,\Delta-1]$}.  
	\smallskip 

	Let  $w_1\in N_{\Delta-1}(r)$ such that $\varphi(rw_1)=\tau$. 
	Since $us_{\mu-1}\in P_r(1,\tau, \varphi) $  and  the color on $us_{\mu-1}$ 
	is not 1 under  every $(L,\varphi)$-stable coloring by Claim~\ref{not1_3}, 
	for every  $(L,\varphi)$-stable $\varphi' \in \CC^{\Delta}(G-rs_1)$  obtained from $\varphi$ 
	through a sequence of  Kempe $(1,*)$-changes not using $r$ or $x$ as endvertices,  it holds that 
	$\varphi'(us_{\mu-1})=\varphi(us_{\mu-1})=\tau$.   Therefore, 
	$u\in P_r(1,\tau, \varphi')$  by  Claim~\ref{not1_3} again. 
	Applying  Lemma~\ref{Lemma:pseudo-fan0} on $L$ with $u$ playing the role of $w$, 
	there exists a sequence of distinct vertices $w_1, \ldots, w_t\in  \{s_{\alpha+1}, \ldots, s_{\Delta-2} \}$ that forms either a stable rotation or a near stable rotation. 

	\smallskip 
	{\noindent \bf \setword{Subcase 1.2.1}{Subcase 1.2.1}:  $w_1, \ldots, w_t$ form a near stable rotation, i.e., $\pbar(w_t)=\alpha+1$.}
	\smallskip 

	In this case, $t\ge 1$. By Claim~\ref{not1_3},  $us_{\mu-1}\in P_r(1,\tau)=P_{z}(1,\tau)$,   for some vertex
	$z\in V(G)\setminus (V(F)\cup \{x,w_1, \ldots, w_t\})$. Assume first that $w_t\ne x$. 
		{ If  $P_z(1,\tau)$ meets $u$ before $s_{\mu-1}$}, we do the following operations:
	\[
	\begin{bmatrix}
	P_{[z,u]}(1,\tau)  & us_{\mu-1} & ux & ur\\
	1/\tau & \tau \rightarrow \mu & \mu \rightarrow \alpha+1 & \alpha+1 \rightarrow 1
	\end{bmatrix}.
	\]
	Denote the new coloring by $\varphi'$. Now $(r,rs_1,s_1,\ldots,s_{\mu-1},rw_1,w_1,\ldots,w_t)$ is a multifan, but $\pbar'(w_t)=\pbar'(r)=\alpha+1$, giving a contradiction to Lemma~\ref{thm:vizing-fan1}~\eqref{thm:vizing-fan1a}.
	{If $P_z(1,\tau)$ meets $s_{\mu-1}$ before $u$}, we do the following operations:
	\[
	\begin{bmatrix}
	P_{[z,s_{\mu-1}]}(1,\tau) & w_1:w_t  & us_{\mu-1} & ux & ur\\
	1/\tau & \text{shift} & \tau \rightarrow \mu & \mu \rightarrow \alpha+1 & \alpha+1 \rightarrow \tau
	\end{bmatrix}.
	\]
	Denote the new coloring by $\varphi'$. Now $(r,rs_1,s_1,\ldots,s_{\mu-1})$ is a multifan, but $\pbar'(s_{\mu-1})=\pbar'(r)=1$, giving a contradiction to  Lemma~\ref{thm:vizing-fan1}~\eqref{thm:vizing-fan1a}.
	
	%
	
Assume now that $w_t= x$. We first consider the case when $t\ge 2$. Note that $P_r(1,\alpha+1)=P_{s_\alpha}(1,\alpha+1)$ and so $r\notin P_x(1,\alpha+1)$. Let  $\varphi_1=\varphi/P_x(1,\alpha+1)$. Then $P_r(1,\tau_{t},\varphi_1)=rx$, where $\tau_t:=\varphi(rw_t)$. We next let  $\varphi_2=\varphi_1/P_{w_{t-1}}(1,\tau_t,\varphi_1)$. Then $\varphi_2$ is $(F,\varphi)$-stable and $\pbar_2(x)=\pbar_2(w_{t-1})=1$. Now doing a $(1, \alpha+1)$-swap at both $x$ and $w_{t-1}$ 
gives back to the previous case when $\pbar(w_t)=\alpha+1$ and $w_t\ne x$
	with $w_{t-1}$ in place of $w_t$.
	
	Thus we assume that $t=1$. Let $\varphi_1=\varphi/P_x(1,\alpha+1)$. Then $P_r(1,\tau,\varphi_1)=rx$. We next let  $\varphi_2=\varphi_1/Q$, where $Q$ is the $(1,\tau)$-chain containing $us_{\mu-1}$ under $\varphi_1$. Then $\varphi_2$ is $(F,\varphi)$-stable, but  $P_{s_{\mu-1}}(1,\mu,\varphi_2)=s_{\mu-1}ux$ ends at $x$, giving a contradiction to Lemma~\ref{thm:vizing-fan1}~\eqref{thm:vizing-fan1b}.

	\smallskip 
	{\noindent  \bf Subcase 1.2.2: $w_1, \ldots, w_t$ form a stable rotation.}
		In this case,  $t\ge 2$.  By Claim~\ref{not1_3},  $us_{\mu-1}\in P_r(1,\tau)=P_{w_t}(1,\tau)$. 
		\smallskip 
		
	{Suppose first that $r\notin P_{s_\alpha}(\alpha+1,\tau)$.}	
	Let $\varphi'=\varphi/P_{s_\alpha}(\alpha+1,\tau)$. Note that $F'=(r,rs_1,s_1,\ldots,s_\alpha,rw_1,w_1,\ldots,w_t)$ is a multifan under $\varphi'$. If the other end of $P_{s_\alpha}(\alpha+1,\tau,\varphi)$ is not $w_t$, then $\pbar'(s_\alpha)=\pbar'(w_t)=\tau$, giving a contradiction Lemma~\ref{thm:vizing-fan1} \eqref{thm:vizing-fan1a} . If the other end of $P_{s_\alpha}(\alpha+1,\tau,\varphi)$ is $w_t$, then $\pbar'(w_t)=\varphi'(ru)=\pbar'(x)=\alpha+1$. Note that $\tau$ is in $\pbar'(F')$ now, and we are back to \ref{Subcase 1.1} when $\mu \prec \tau \prec \alpha+1$
	with $F'$ in the place of $F$ (so $L$ is understood to be defined with respect to $F'$ too).

	{Assume  now that $r\in P_{s_\alpha}(\alpha+1,\tau)$}. We consider the following three  cases. 
	
		\smallskip 
	{\noindent  \bf Subcase A: $s_\alpha$ and $w_t$ are $(\alpha+1,\tau)$-linked.} 
		\smallskip 
		
	Let $\varphi'=\varphi/P_x(\alpha+1,\tau)$. Then in the new coloring, $P_{s_{\mu-1}}(\mu,\tau,\varphi')=s_{\mu-1}ux$. Let $\varphi''=\varphi'/P_{s_{\mu-1}}(\mu,\tau,\varphi')$. Then $F^*=(r,rs_1, s_1,\dots, rs_{\mu-1}, s_{\mu-1}, rw_1, w_1, \ldots, rw_t, w_t)$ is a multifan with respect to $\varphi''$. 
	However, $\pbar''(s_{\mu-1})=\pbar''(w_t)=\tau$, showing a contradiction to Lemma~\ref{thm:vizing-fan1} \eqref{thm:vizing-fan1a} that  $V(F^*)$
	is $\varphi''$-elementary.
	
		\smallskip 
	{\noindent \bf \setword{Subcase B}{Subcase 1.2.2.3} : $s_\alpha$ and $w_t$ are $(\alpha+1,\tau)$-unlinked, but 
		$s_\alpha$ and $x$ are $(\alpha+1,\tau)$-linked.} 
		\smallskip 
	
	Since  $r\in P_{s_\alpha}(\alpha+1,\tau)$, 
	letting $\varphi'=\varphi/P_{w_t}(\alpha+1,\tau)$ reducing  the problem to \ref{Subcase 1.2.1}.

		\smallskip 
	{\noindent  \bf Subcase C: $s_\alpha$ is   $(\alpha+1,\tau)$-unlinked with neither $w_t$ nor $x$.}
		\smallskip 
	
	We may assume that $x$ and $w_t$ are $(\alpha+1, \tau)$-linked. For otherwise, 
	letting $\varphi'=\varphi/P_{w_t}(\alpha+1,\tau)$ gives back to \ref{Subcase 1.2.1} again.
		Recall that $r\in P_{s_\alpha}(\alpha+1,\tau)$. If $P_{s_\alpha}(\alpha+1, \tau)$ meets $w_1$ before $s_{\mu-1}$, we shift  from $w_1$ to $w_t$. This yields a new coloring $\varphi'$ such that 
	$r\not\in P_{s_\alpha}(\alpha+1,\tau,\varphi') $, and so we are back to the first subcase of Subcase 1.2.2 when $r\notin P_{s_\alpha}(\alpha+1,\tau,\varphi)$. If $P_{s_\alpha}(\alpha+1, \tau)$ meets $s_{\mu-1}$ before $w_1$, then shift  from $w_1$ to $w_t$ yields a new coloring $\varphi'$ such that
	$s_\alpha$ and $x$ are $(\alpha+1,\tau)$-linked with respect to $\varphi'$, which reduces the problem to 
	\ref{Subcase 1.2.2.3}.

		\smallskip 
{\noindent \bf Case 2: $u\sim s_\mu$}.  
	Let $\varphi(us_\mu)=\tau$. 	Note that $\tau\ne \mu, \mu+1, \alpha+1$. 
			\smallskip 
			
	\begin{CLA}\label{cla3.3}
		Either $\tau=\Delta$ or $\tau$ is a 2-inducing color with $\tau\prec \mu$.
	\end{CLA}
	\proof[Proof of Claim~\ref{cla3.3}] 
	Let $\varphi'$ be the coloring obtained from $\varphi$ by uncoloring $rs_{\mu}$, shift from $s_2$ to $s_{\mu-1}$ and coloring $rs_1$ by 2. Then $F'=(r,rs_{\mu},s_{\mu},rs_{\mu+1},s_{\mu+1},\ldots,s_{\alpha},rs_{\mu-1},s_{\mu-1},\ldots,s_1)$ is a multifan under $\varphi'$,  where  $\pbar'(s_{\mu})=\{\mu,\mu+1\}$, $\varphi'(ru)=\pbar'(x)=\alpha+1$ is the last $(\mu+1)$-inducing color,  $\varphi'(ux)=\mu$, and $u \sim s_{\mu}$. We can further assume that $F'$ is typical by remaining colors and vertices.  
	Thus by Theorem~\ref{Lem:2-non-adj1*} (1), $\tau$ is a $\mu$-inducing color with respect to $\varphi'$ and $F'$. So with respect to the original coloring $\varphi$ and $F$, we have either $\tau=\Delta$ or $\tau$ is a 2-inducing color with $\tau\prec \mu$. \qed
	
		\smallskip 
	{\noindent \bf Subcase 2.1: $\tau$ is a 2-inducing color with $\tau\prec \mu$}.  
		\smallskip 

	\smallskip 
	
	By Lemma~\ref{Lemma:extended multifan} \eqref{Evizingfan-e},  $r\in P_{s_\alpha}(\alpha+1,\tau)=P_{s_{\tau-1}}(\alpha+1,\tau)$.  
	Let $\varphi'=\varphi/P_x(\alpha+1,\tau)$. 
	Then $\pbar'(x)=\tau$. 
	It must  be still the case that $u\in P_r(1,\tau, \varphi')=P_{s_{\tau-1}}(1,\tau, \varphi')$. 
	For otherwise, swapping colors along $P_x(1,\tau,\varphi')$ and the $(1,\tau)$-chain containing $u$ (only swap once if the two chains are the same)
	gives a coloring $\varphi''$ such that $P_r(1,\mu,  \varphi'')=rs_\mu ux$, 
	showing a contradiction to Lemma~\ref{thm:vizing-fan1}~\eqref{thm:vizing-fan1b} 
	that $r$ and $s_{\mu-1}$ are $(\mu,1)$-linked with respect to $\varphi''$. 
	Let $\varphi^*=\varphi'/P_x(1,\tau,\varphi')$. Now $\pbar^*(x)=1$ and $u\in P_r(1,\tau,\varphi^*)$. 	
If $P_{s_{\tau-1}}(1,\tau, \varphi^*)$ meets $u$ before $s_{\mu}$, then we do the following operations:
	\[
	\begin{bmatrix}
	s_{\tau}: s_{\mu-1}  &  rs_\mu &s_\mu u & ux & P_{[s_{\tau-1}, u]}(1,\tau,\varphi^*) \\
	\text{shift} & \mu  \rightarrow \tau & \tau \rightarrow \mu & \mu\rightarrow 1 & 1/\tau
	\end{bmatrix}.
	\]
	Denote the new coloring by $\varphi''$. Now $(r,rs_1,s_1,\ldots,s_{\tau-1})$ is a multifan, but $\pbar''(s_{\tau-1})=\pbar''(r)=1$, giving a contradiction to  Lemma~\ref{thm:vizing-fan1}~\eqref{thm:vizing-fan1a}.
Thus $P_{s_{\tau-1}}(1,\tau, \varphi^*)$ meets $s_\mu$ before $u$.  
We do the following operations:
\[
\begin{bmatrix}
P_{[s_{\tau-1}, s_\mu]}(1,\tau,\varphi^*)  & us_\mu  & ur & rs_\mu  & s_2:s_{\mu-1}  & s_{\mu+1}:s_\alpha\\
1/\tau &\tau\rightarrow \mu & \alpha+1\rightarrow \tau & \mu \rightarrow \mu+1 &\text{shift} &\text{shift}
\end{bmatrix}.
\]

Based on the resulting coloring from above, we color $rs_1$ by 2 if $\tau\ne 2$ and by 1 if $\tau=2$, and uncolor $ux$.
Denote the new coloring by $\varphi''$. 
 Now $F^*=(u,ux, x, us_\mu, s_\mu )$ 
is a multifan with respect to $ux$ and $\varphi''$. 
However, the color 1 is missing at both $x$ and $s_\mu$,
showing a contradiction to Lemma~\ref{thm:vizing-fan1} \eqref{thm:vizing-fan1a}.

		\smallskip 
	{\noindent \bf  Subcase 2.2: $\tau=\Delta$}.  
		\smallskip 
		
	By Lemma~\ref{Lemma:extended multifan} \eqref{Evizingfan-d}, we know that $r\in P_{s_\alpha}(\alpha+1,\Delta)=P_{s_1}(\alpha+1,\Delta)$.  
	Let $\varphi'=\varphi/P_x(\alpha+1,\Delta)$. 
	Then $\pbar'(x)=\Delta$. 
	It must  be still the case that $u\in P_r(1,\Delta, \varphi')=P_{s_{1}}(1,\Delta, \varphi')$. 
	For otherwise, swapping colors along $P_x(1,\Delta,\varphi')$ and the $(1,\Delta)$-chain containing $u$ (only swap once if the two chains are the same) gives a coloring $\varphi''$ such that $P_r(1,\mu,  \varphi'')=rs_\mu ux$, showing a contradiction to Lemma~\ref{thm:vizing-fan1}~\eqref{thm:vizing-fan1b} 
	that $r$ and $s_{\mu-1}$ are $(1,\mu)$-linked with respect to $\varphi''$. 
	If $P_{s_1}(1,\Delta, \varphi')$ meets  $s_{\mu}$ before $u$, then  we do the following operations:
	\[
	\begin{bmatrix}
	P_{[s_1, s_\mu]}(1,\Delta,\varphi') & rs_{\mu} & s_{\mu}u &ux\\
	1/\Delta & \mu\rightarrow 1 & \Delta \rightarrow \mu & \mu\rightarrow \Delta
	\end{bmatrix}.
	\]
	Denote the new coloring by $\varphi''$. Now $(r,rs_1,s_1,\ldots,s_{\mu-1})$ is a multifan, but $\pbar''(s_{\mu-1})=\pbar''(r)=\mu$, giving a contradiction to Lemma~\ref{thm:vizing-fan1} \eqref{thm:vizing-fan1a}.

Thus $P_{s_1}(1,\Delta, \varphi')$  meets $u$ before $s_\mu$.  
We do the following operations:
\[
\begin{bmatrix}
P_{[s_1, u]}(1,\Delta,\varphi')  & us_\mu  & ur & rs_\mu  & s_2:s_{\mu-1}  & s_{\mu+1}:s_\alpha\\
1/\Delta &\Delta\rightarrow \mu & \alpha+1\rightarrow 1 & \mu \rightarrow \mu+1 &\text{shift} &\text{shift}
\end{bmatrix}.
\]
Based on the resulting coloring from above, we color $rs_1$ by 2 and uncolor $ux$.
Denote the new coloring by $\varphi''$. 
Now $F^*=(u,ux, x, us_\mu, s_\mu )$ 
is a multifan with respect to $ux$ and $\varphi''$. 
However,  the color $\Delta$ is missing at both $x$ and $s_\mu$,
showing a contradiction to Lemma~\ref{thm:vizing-fan1} \eqref{thm:vizing-fan1a}. 
The proof of Theorem~\ref{Lem:2-non-adj2} is now complete.
\end{proof}

 \section*{Acknowledgements}
Guantao Chen was supported by NSF grants DMS-1855716 and DMS-2154331; 
Guangming Jing was supported by NSF grant DMS-2246292; 
and Songling Shan was partially supported by 
NSF grant DMS-2345869. 

\end{document}